\documentclass[psamsfont]{amsart}
\usepackage[utf8]{inputenc}
\usepackage{graphicx}
\usepackage[dvips]{epsfig}
\usepackage{pinlabel}
\usepackage{amsmath}
\usepackage{amsfonts}
\usepackage{latexsym}
\usepackage{mathabx}
\usepackage{amssymb}
\usepackage[usenames,dvipsnames]{color}
\usepackage{amsthm}
\usepackage[all]{xypic}
\usepackage{enumitem}
\usepackage[breaklinks=true]{hyperref}

\input xy
\xyoption{all}

\author{Baptiste Chantraine}
\author{Vincent Colin}
\author{Georgios Dimitroglou Rizell}
\address{Universit\'e de Nantes, France.}
\email{baptiste.chantraine@univ-nantes.fr}
\email{vincent.colin@univ-nantes.fr}
\address{Uppsala University, Sweden.}
\email{georgios.dimitroglou@math.uu.se}

\thanks{The first author is partially supported by the ANR project COSPIN (ANR-13-JS01-0008-01). The first and second authors are supported by the ERC grant Geodycon. The third author was supported by the grants KAW 2013.0321 and KAW 2016.0446 from the Knut and Alice Wallenberg Foundation.}

\theoremstyle{plain}
\newtheorem{Thm}{Theorem}[section]
\newtheorem{Rem}[Thm]{Remark}
\newtheorem{Prop}[Thm]{Proposition}
\newtheorem{Lem}[Thm]{Lemma}
\newtheorem{Cor}[Thm]{Corollary}

\theoremstyle{remark}
\newtheorem{defn}[Thm]{Definition}

\newtheorem{Ex}[Thm]{Example}

\DeclareMathAlphabet{\mathdj}{U}{msb}{m}{n}
\newcommand{\R}{\ensuremath{\mathdj{R}}}
\newcommand{\T}{\ensuremath{\mathdj{T}}}

\newcommand{\Z}{\ensuremath{\mathdj{Z}}}

\newcommand{\Cth}{\operatorname{Cth}}

\newcommand{\id}{\operatorname{Id}}
\newcommand{\im}{\operatorname{im}}

\newcommand{\OP}{\operatorname}

\begin{document}
\title{Positive Legendrian isotopies and Floer theory}
\thispagestyle{empty}
\maketitle

\begin{abstract}
Positive loops of Legendrian embeddings are examined from the point of view of Floer homology of Lagrangian cobordisms. This leads to new obstructions to the existence of a positive loop containing a given Legendrian, expressed in terms of the Legendrian contact homology of the Legendrian submanifold. As applications, old and new examples of orderable contact manifolds are obtained and discussed. We also show that contact manifolds filled by a Liouville domain with non-zero symplectic homology are \emph{strongly} orderable in the sense of Liu.
\end{abstract}

\section{Introduction}
\label{sec:introduction}
Since the groundbreaking work \cite{ElPo_Order} by Eliashberg--Polterovich, the notion of orderability has played an important role in the study of contact geometry. Recall that a contact manifold is \emph{orderable} if and only if it admits no positive loop of contactomorphisms which is contractible (amongst loops of arbitrary contactomorphisms). In \cite{ElKiPo} it was shown that a large class of subcritically fillable contact manifolds are \emph{non-orderable}, including the standard odd-dimensional contact spheres, while it follows from Givental's non-linear Maslov index in \cite{Givental} that the standard contact structures on the odd-dimensional real projective spaces are orderable. In some cases orderability is known to imply the existence of an unbounded bi-invariant metric on the space of contactomorphisms (see \cite{fraporo}), having a number of important consequences; see~\cite{Sandon1}, \cite{Colin_Sandon_Osci}, and \cite{Sandon2} for more details.

There are conditions in terms of Floer homology that imply that a contact manifold is orderable. Notably, in \cite{AlMe} Albers--Merry showed that if the Rabinowitz Floer homology of a Liouville domain admits a non-trivial spectrally finite class, then its contact boundary must be orderable. Here we strengthen this result by showing that the non-vanishing of Rabinowitz Floer homology is sufficient  (see Theorem \ref{thm: strong} combined with \cite[Theorem 13.3]{Ritter_TQFT} by which Rabinowitz Floer homology vanishes if and only if symplectic homology does).

In this article we consider this notion from the relative point of view, i.e.~from the perspective of Legendrian submanifolds. Let $(M,\xi=\ker \alpha)$ be a co-oriented contact manifold. Through the paper we always assume that the contact form $\alpha$ induces the given co-orientation. A Legendrian isotopy $\{\Lambda^s\}$, $s \in [0,1]$, (where $\Lambda^s:\Lambda\hookrightarrow M$ is a smooth family of Legendrian embeddings) is \textit{positive} if for every $q\in \Lambda$ and $s \in [0,1]$ we have
\begin{equation}
  \label{eq:3}
  H(s,q):=\alpha\left(\dot{\Lambda}^{s}(q)\right)>0.
\end{equation}
This definition only depends on the co-orientation of $\xi$ and not on the parametrisation of $\Lambda^s$, nor on the choice of contact form $\alpha$ (as long as it induces the positive co-orientation). The most basic example of a positive isotopy is the displacement of a Legendrian submanifold by the Reeb flow associated to a choice of contact form. A positive isotopy for which $\Lambda^0=\Lambda^1$ will be called a \textit{positive loop of Legendrians containing $\Lambda^0$}. Since Legendrian isotopies are realised by ambient contact isotopies, if $\Lambda$ is in a positive loop of Legendrians and $\Lambda'$ is Legendrian isotopic to $\Lambda$, then $\Lambda'$ is also sits in a positive loop. (The reason is that $\phi(\Lambda_t)$ is positive when $\Lambda_t$ is a positive loop and $\phi$ is a contactomorphism preserving the coorientation of the contact structure.)

Our goal is finding new obstructions for the existence of positive loops as well as contractible positive loops containing a given Legendrian submanifold. By a \emph{contractible positive loop of Legendrians} we mean a positive loop which is contractible as a loop of Legendrian embeddings. The obstructions that we find are in terms of Legendrian contact homology (LCH for short) as well as wrapped Floer cohomology. These are two related symplectic invariants that algebraically encode counts of different types of pseudoholomorphic curves. Obviously, in the case when the contact manifold is \emph{not} orderable, each of its Legendrian submanifolds lives in a contractible positive loop of Legendrians. Thus, when our obstructions apply, they can be used as a condition that ensures orderability of the ambient contact manifold.

\subsection{Previous results}

Concerning positive loops of Legendrian submanifolds, the second author together with Ferrand and Pushkar used generating family techniques to show the following.
\begin{Thm}[Colin--Ferrand--Pushkar \cite{CoFePu}]\label{thm:CFP}
Let $Q$ denote a smooth, not necessarily closed, manifold.
\begin{enumerate}
\item There exists no positive loop of Legendrians containing the zero-section in $(\mathcal{J}^1Q,\xi_{\OP{std}})$.
\item If the universal cover of $Q$ is $\mathbb{R}^n$ then there exists no positive loop of Legendrians containing a Legendrian fibre of the canonical projection $\pi:S(T^*Q)\rightarrow Q$ in $(S(T^*Q),\xi_{\OP{std}})$.
\end{enumerate}
\end{Thm}

The analogue of the above theorem cannot be expected to hold for general Legendrian submanifolds, as is shown by the following example.
\begin{Ex}
The global contact isotopy $(q,p,z)\mapsto (q+t,p,z)$ of $\mathcal{J}^1S^1$, generated by the contact Hamiltonian $H(q,p,z)=p$ for the standard contact form, is a loop of contactomorphisms starting and ending at the identity. Moreover, it is positive when restricted to either of the two Legendrian knots with the front diagrams shown in Figure \ref{fig:posp}: the positive stabilisation $S^+(0_{S^1})$ of the zero-section $0_{S^1} \subset \mathcal{J}^1S^1$ on the left, as well as the representative of the standard Legendrian unknot shown on the right.
\end{Ex}

\begin{figure}[ht!]
  \centering
  \includegraphics[height=5cm]{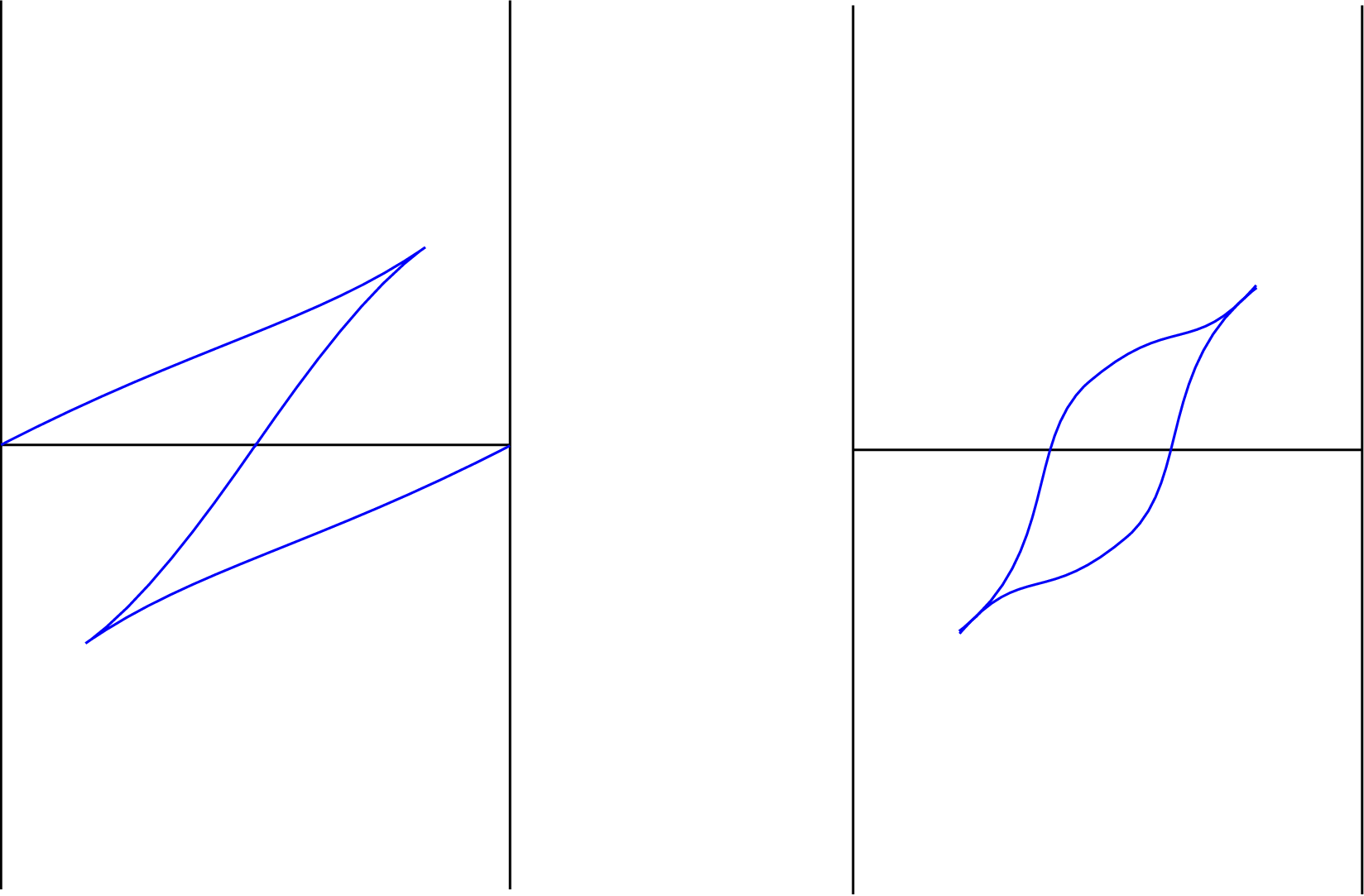}
  \caption{Fronts of Legendrian submanifolds in the subspace $\{p>0\} \subset \mathcal{J}^1S^1$. On the left: the positive stabilisation $S^+(0_{S^1})$ of the zero-section. On the right: a representative of the standard Legendrian unknot.}
  \label{fig:posp}
\end{figure}

\begin{Ex}
\label{remWeinsteinposisot} Since there is a contact embedding of a neighbourhood of the zero-section of $\mathcal{J}^1S^1$ into any three-dimensional contact manifold, any Legendrian unknot can be seen to sit inside a positive loop of Legendrians. Such a positive Legendrian loop inside a Darboux ball, which moreover is contractible, is shown in Figure \ref{fig:trivposisoto}. The isotopy is indeed positive if, in the parts I and III of the isotopy, the translation is in a direction whose slope is less than the maximal slope of the front projection of the unknot. If the $z$-coordinate is decreased sufficiently during step I and III of the isotopy, the rotation of the front taking place in steps II and IV can be made positive. Finally, observe that this loop is contractible amongst loops of Legendrian submanifolds
\end{Ex}

\begin{figure}[ht!]
  \centering
\labellist
\pinlabel{I} at 225 280
\pinlabel{III} at 460 280
\pinlabel{II} at 630 250
\pinlabel{IV} at 40 250
\endlabellist
  \includegraphics[height=5cm]{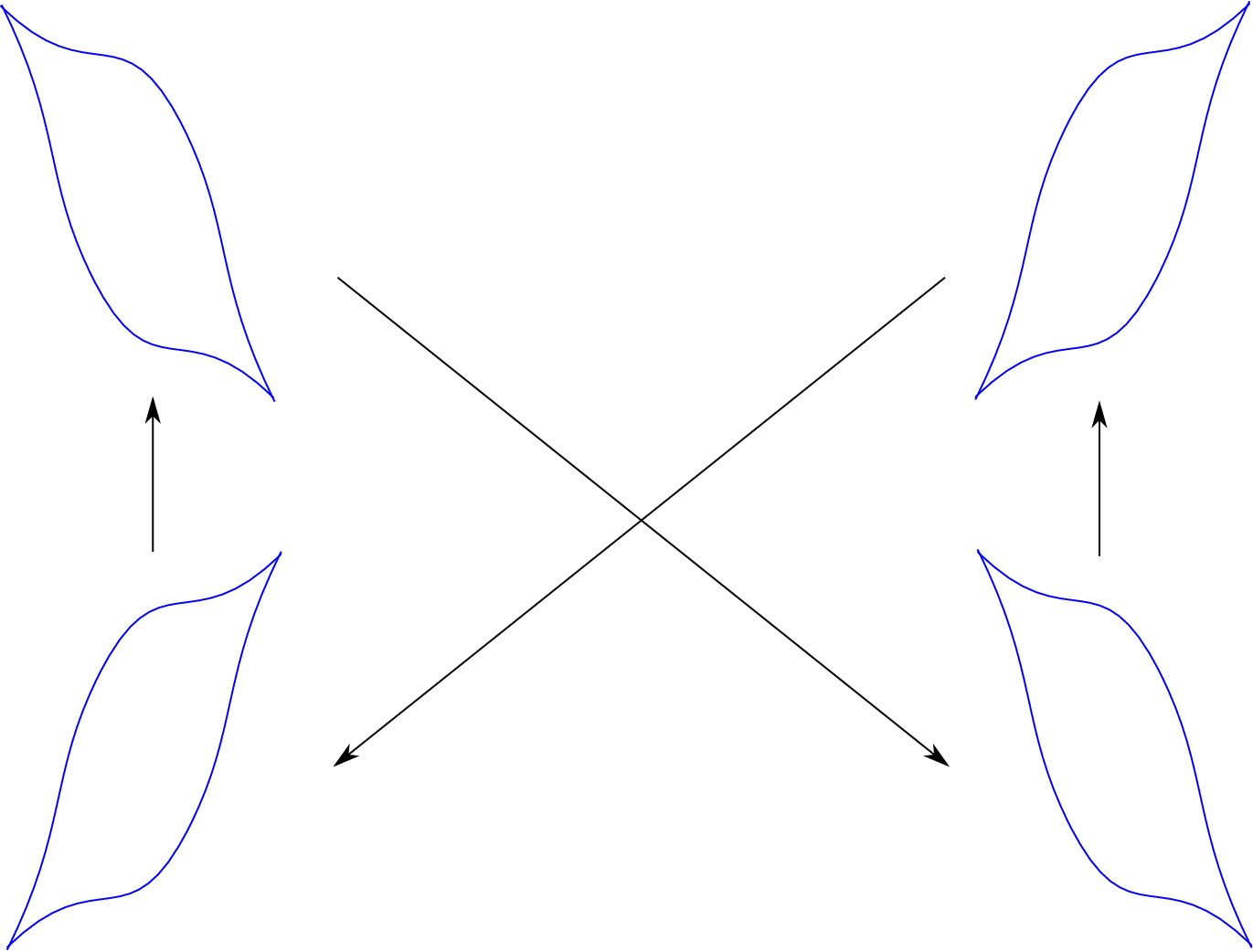}
  \caption{A positive isotopy of the standard Legendrian unknot inside $\mathcal{J}^1\R$ that, moreover, is contractible.}
  \label{fig:trivposisoto}
\end{figure}
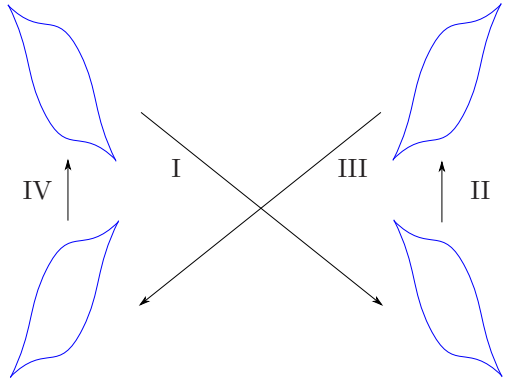
The above constructions of positive loops of Legendrians are generalised in Liu's work \cite{Liu,Liu_paper}:
\begin{Thm}[Liu \cite{Liu_paper}]\label{thm:liu}
 Every loose Legendrian submanifold is contained in a contractible positive loop of Legendrians.
\end{Thm}
We also refer to the more recent work \cite{Presas} by Pancholi--P\'{e}rez--Presas.

Part (2) of Theorem \ref{thm:CFP} was generalised by Chernov--Nemirovski to more general spaces of contact elements. Notably, they showed the following:
\begin{Thm}[Chernov--Nemirovsky \cite{Cher_Nem}, \cite{Cher_Nem_2}] \label{thm:CN} 
When $Q$ is a connected open manifold, there are no positive loops of Legendrians containing a Legendrian fibre of $S(T^*Q)$. For general $Q$ there is no positive loop containing the Legendrian fibre in $S(T^*Q)$ which is contractible amongst loops of Legendrians.
\end{Thm}
Also, see the work \cite{Lucas} by Dahinden for obstructions in certain cases when the universal cover is not open.

In the second case of Theorem \ref{thm:CN} the contractibility condition is essential. Indeed, if $S^n$ denotes the round sphere in $\R^{n+1}$, the Reeb flow on the unit cotangent bundle $S(T^*S^n)$ for the contact form $p dq$ given by the restriction of the Liouville form corresponds to the geodesic flow. In this way we see that any Legendrian submanifold sits inside a positive loop. However, for the Legendrian fibre, such a loop fails to be contractible in the space of Legendrian embeddings by the above results.

\begin{Rem}
In \cite{GKS} Guillermou--Kashiwara--Shapira reprove Theorem \ref{thm:CN} using the method of microsupports of constructible sheaves. In the present paper we consider a class of contact manifolds that is strictly larger than jet-spaces and spaces of contact elements, i.e.~contact manifolds where the methods based upon the microsupport of sheaves are not yet applicable.
\end{Rem}

\subsection{Results}

In this paper we obtain generalisations of the above results. We say that a contact manifold $(M,\xi)$ is \emph{hypertight} if it admits a (possibly degenerate, see Section \ref{sec:ht}) contact form having no contractible periodic Reeb orbits; the latter contact form will be called \emph{hypertight} as well. Likewise, a Legendrian submanifold $\Lambda \subset (M,\xi)$ is called \emph{hypertight} if there is a hypertight contact form for which $\Lambda$, moreover, has no contractible Reeb chords (i.e.~Reeb chords in the homotopy class $0 \in \pi_1(M,\Lambda)$).

The contactisation of a Liouville manifold $(P,d\theta)$ is the (non-compact) hypertight contact manifold $(P \times \R,\alpha_{\OP{std}})$, $\alpha_{\OP{std}}=\theta+dz$, with $z$ denoting the coordinate on the $\R$-factor. In such manifolds, lifts of exact Lagrangians in $P$ are particular cases of hypertight Legendrians. 
\begin{Ex}
The archetypal example of a contact manifold of the above form is the jet-space $\mathcal{J}^1Q$ of a smooth manifold, which is the contactisation of $(T^*Q,-d\lambda_Q)$ for the Liouville form $\lambda_Q$ (note the sign in our convention!). The zero-section $0_Q \subset T^*Q$ is an exact Lagrangian submanifold which lifts to the zero-section of $\mathcal{J}^1Q$; this is an embedded hypertight Legendrian submanifold.
\end{Ex}

In the present paper, when talking about hypertight contact manifolds, we also assume that outside a compact set they are equivalent to the contactisation of a Liouville manifold. The hypertightness assumption is mainly a technical one, and we refer to Section \ref{rem_pardon} below for an explanation of what we expect to hold more generally.

The first result is an obstruction to the existence of a positive Legendrian loop expressed in terms of Legendrian contact cohomology. This is a Legendrian isotopy invariant originally defined in \cite{Chekanov_DGA_Legendrian} by Chekanov, and also sketched in \cite{Eliashberg_&_SFT} by Eliashberg--Givental--Hofer. The theory has been rigorously defined in a wide range of contact manifolds; see \cite{Chekanov_DGA_Legendrian} for one-dimensional Legendrians and \cite{LCHgeneral} for the case of a general contactisation.

Let $\Lambda,\Lambda' \subset (P \times \R,\alpha_{\OP{std}})$ be Legendrian submanifolds of the contactisation of a Liouville domain, each of which having a Chekanov--Eliashberg algebra admitting augmentations $\varepsilon$ and $\varepsilon'$, respectively. As described in Section \ref{sec:lch} below, there is an associated linearised Legendrian contact cohomology complex $LCC_{\varepsilon,\varepsilon'}^*(\Lambda,\Lambda')$ generated by Reeb chords from $\Lambda'$ to $\Lambda$ (observe the order!). The homotopy type of the complex $LCC_{\varepsilon,\varepsilon'}^*(\Lambda,\Lambda')$ can be seen to only depend on the Legendrian isotopy class of the link $\Lambda \cup \Lambda'$ and the augmentations chosen; see e.g.~\cite{Duality_EkholmetAl}.
\begin{Thm}
\label{thm:liouville}
Assume that $LCC_{\varepsilon,\varepsilon'}^*(\Lambda,\Lambda')$ is not acyclic for $\Lambda,\Lambda' \subset (P\times\R,\theta+dz)$. Then $\Lambda$ is not part of a positive loop of Legendrian submanifolds of the complement $M \setminus \Lambda'$. Under the additional assumption that
$$\min_\Lambda z > \max_{\Lambda'} z$$
is satisfied, then $\Lambda$ is not contained inside \emph{any} positive loop of Legendrians.
\end{Thm}

\begin{Rem} Under the stronger assumption $\min_\Lambda z > \max_{\Lambda'} z$, the homotopy type of the complex $LCC_{\varepsilon,\varepsilon'}^*(\Lambda,\Lambda')$ can be interpreted as also being invariant under Legendrian isotopy of each $\Lambda$ and $\Lambda'$ separately with the following caveat: we must first make the two Legendrian isotopies disjoint by translating the second family $\Lambda_t'$ very far in the negative $z$-direction (i.e.~by applying the negative Reeb flow, which is a contact form preserving isotopy).
\end{Rem}
In the case when $\min_\Lambda z > \max_{\Lambda'} z$ is satisfied, the complex $LCC_{\varepsilon,\varepsilon'}^*(\Lambda,\Lambda')$ can be interpreted a version of the Floer homology complex
$$CF(\Pi_{\OP{Lag}}(\Lambda),\Pi_{\OP{Lag}}(\Lambda')) \subset (P,d\theta)$$
for a pair of exact Lagrangian immersions. We refer to \cite{AkahoJoyce} for general treatment of Lagrangian Floer homology in the immersed case.

When neither of $\Lambda$ nor $\Lambda'$ have any Reeb chords their Lagrangian projections are exact Lagrangian \emph{embeddings}. In this case, both Legendrian submanifolds have a unique canonical augmentation and, given that $\min_\Lambda z > \max_{\Lambda'}z$, there is a canonical identification with the classical Lagrangian Floer homology complex $CF^*(\Pi_{\OP{Lag}}(\Lambda),\Pi_{\OP{Lag}}(\Lambda'))$ defined by Floer \cite{FloerHFlag}. In the case when $\Lambda'$ moreover is obtained from $\Lambda$ by an application of the negative Reeb flow followed by a sufficiently $C^1$-small Legendrian perturbation, Floer's original computation in \cite{FloerHFlag} shows that there is an isomorphism
$$LCC_{\varepsilon,\varepsilon'}^*(\Lambda,\Lambda')=CF^*(\Pi(\Lambda),\Pi(\Lambda'))=C_*(\Lambda)$$
of the Floer homology complex and Morse homology complex. From this we now conclude that
\begin{Cor}
The Legendrian lift $\Lambda_L \subset (P \times \R,\alpha_{\OP{std}})$ of an exact Lagrangian embedding $L \subset (P,d\theta)$ is not contained in a positive loop of Legendrians.
\end{Cor}

Theorem \ref{thm:liouville} is proven using Lagrangian Floer homology for Lagrangian cobordisms (which also goes under the name `Cthulhu homology'), which was developed in \cite{cthulhu} by the first and third authors together with Ghiggini and Golovko. In particular, we use this theory to produce the crucial long exact sequence in Theorem \ref{thm:les} below. By a variation of these techniques we are also able to establish the following related result.

\begin{Thm} \label{thm:hypertight}
A hypertight Legendrian submanifold $\Lambda \subset (M,\alpha)$ of a closed hypertight contact manifold is not contained inside a contractible positive loop of Legendrians.
\end{Thm}

\begin{Ex}
There are indeed examples of hypertight Legendrian submanifolds contained in a positive non-contractible loop. Consider e.g.~the conormal lift of the simple closed curve $S^1 \times \{ \theta\} \subset \T^2$ to a Legendrian knot inside $S(T^*\T^2)$. The Reeb flow for the flat metric on $\T^2$ restricted to this conormal lift induces a positive loop which is not contractible by mere topological reasons.
\end{Ex}
The above example should be contrasted with the following result, whose proof is similar to that of Theorem \ref{thm:hypertight}.
\begin{Thm}
\label{thm:hypertight2}
Let $\Lambda \subset (M,\alpha)$ be a hypertight Legendrian submanifold of a closed contact manifold having vanishing Maslov number. Further assume that, for every pair $\gamma_1,\gamma_2$ of Reeb chords on $\Lambda$ in the same homotopy class, the Conley--Zehnder indices satisfy either $\OP{CZ}(\gamma_1)-\OP{CZ}(\gamma_2)=0$ or $|\OP{CZ}(\gamma_1)-\OP{CZ}(\gamma_2)|>\dim \Lambda$. Then $\Lambda$ is not contained in a positive loop of Legendrians.
\end{Thm}

Examples of Legendrian manifolds satisfying the hypothesis of the previous theorem are given by cotangent fibres in the space of contact elements of manifolds with non-positive curvature.
\begin{Cor}[Chernov--Nemirovsky \cite{Cher_Nem}]
The unit cotangent fibre of a manifold with non-positive sectional curvature is not contained in a positive loop of Legendrians.
\end{Cor}

The ideas used for proving the above results can also be applied in the case when the Legendrian submanifold admits an exact Lagrangian filling. In this case, instead of the Floer theory from \cite{cthulhu}, ordinary wrapped Floer cohomology can be used.
\begin{Thm}\label{thm:filling}
If a Legendrian $\Lambda \subset (M,\xi)$ admits an exact Lagrangian filling $L \subset (X,\omega)$ inside a Liouville domain with contact boundary $(\partial X=M,\xi)$, such that the wrapped Floer cohomology of $L$ is nonvanishing (for some choice of coefficients), then $\Lambda$ is not contained inside a contractible positive loop of Legendrians.
\end{Thm}
We remind the reader that, in the case when we can exclude a Legendrian submanifold from being contained in a positive loop, it also follows that the ambient contact manifolds involved are orderable in the sense of \cite{ElPo_Order}: non-orderability would imply that any Legendrian is in a contractible positive loop of Legendrians.
\begin{Rem}
Applying the previous theorem to the Lagrangian diagonal in a symplectic product allows us to show (strong) orderabilty of $(M,\xi)$ in the case when it admits a filling with non-vanishing symplectic homology (see Theorem \ref{thm: strong}), thus strengthening \cite[Corollary 1.3]{AlMe}. In the hypertight case, orderability of the contact manifold is proved in \cite{AlFuMe}.
\end{Rem}
We notice also that we can state the following consequence, relying on Bourgeois--Ekholm--Eliashberg \cite{EffectLegendrian}:

\begin{Cor} Let $(M_+,\xi_+)$ obtained by performing a contact surgery along a Legendrian link $\Lambda \subset (M_-,\xi_-)$ of spheres where
\begin{itemize}
\item $(M_-,\xi_-)$ is the boundary of a subcritical Weinstein domain,
\item the Legendrian contact homology DGA of each component of $\Lambda$ is not acyclic.
\end{itemize}
Then there is no contractible positive Legendrian loop containing a Legendrian co-core sphere  inside $(M_+,\xi_+)$ created by the surgery.
\end{Cor}

Indeed as predicted in \cite{EffectLegendrian}, the wrapped Floer cohomology of the co-core of a attaching handle is isomorphic to the Legendrian contact homology of $\Lambda$. Hence, Theorem \ref{thm:filling} shows that the co-core spheres are not contained in a positive loop of Legendrians when $LCH(\Lambda)\not =0$.

Our methods also apply to prove \emph{strong} orderability of some contact manifolds by passing to contact products (see below for the definition). Let $(M,\xi =\ker \alpha )$ be a contact manifold. The {\it contact product} is the contact manifold $(M\times M\times \R ,\alpha_1 -e^t \alpha_2 )$, where $\alpha_i$ is the pullback of $\alpha$ by the projection $\pi_i : M\times M \times \R \to M$ on the $i$th factor, $i=1,2$. If $\phi$ is a contactomorphism of $(M,\xi)$ with $\phi^* \alpha =e^{g(t)} \alpha$, then the {\it graph} $\Delta_\phi$ of $\phi$ is the Legendrian submanifold $\{ (x,\phi (x), g(x)), \: x\in M\}$ of the contact product.
To a contact isotopy $(\phi_t)_{t\in [0,1]}$ of $(M,\xi)$ we can thus associate a Legendrian isotopy $(\Delta_{\phi_t} )_{t\in [0,1]}$ starting from the {\it diagonal} $\Delta_{\id}=\{ (x,x,0)\}$. When $(\phi_t)_{t\in [0,1]}$ is positive, then $(\Delta_{\phi_t} )_{t\in [0,1]}$ is negative. Following Liu, we can say that $(M,\xi)$ is {\it strongly orderable} whenever there is no contractible positive loop of Legendrians based at the diagonal in the contact product. In that case, we can endow the universal cover of the identity component of the group of contactomorphisms of $(M,\xi)$ with a partial order, by saying that $[(\phi_t)_{t\in [0,1]}] \leq[(\psi_t)_{t\in [0,1]}]$ if there exists a positive path of Legendrians from $\Delta_{\phi_1}$ to $\Delta_{\psi_1}$ which is homotopic to the concatenation of the opposite of $(\Delta_{\phi_t})_{t\in [0,1]}$ together with $(\Delta_{\psi_t})_{t\in [0,1]}$. This is possibly a different notion from Eliashberg--Polterovich's order since the latter also requires paths of Legendrians to stay amongst \emph{graphs of contactomorphisms}.

In \cite{Liu_paper}, Liu proved that if $(M,\xi )$ is overtwisted (in the sense of Borman--Eliashberg--Murphy \cite{ExistenceOvertwisted}), then the contact product $(M\times M\times \R ,\alpha_1 -e^t \alpha_2 )$ is also overtwisted and its diagonal is a loose Legendrian. Thus the diagonal is the base point of a contractible positive loop by Theorem \ref{thm:liu} and $(M,\xi )$ is not strongly orderable.

Here we prove the following.

\begin{Thm} \label{thm: strong} If $(M,\xi=\ker\alpha)$ is the contact boundary of a Liouville domain $(W,\omega=d\alpha)$ whose symplectic cohomology does not vanish for some choice of coefficients, i.e.~$SH^*(W,\omega)\neq 0$, then $(M,\xi)$ is strongly orderable.
\end{Thm}

Theorem \ref{thm:hypertight} also relates to strong orderability of closed hypertight contact manifolds. Namely, by using Zena\"{i}di's compactness result \cite[Theorem 5.3.9]{Zenaidi}, one can extend Theorem \ref{thm:hypertight} to the Legendrian diagonal $\Delta_{\id}$ in $M\times M\times \mathbb{R}$. Indeed, when $(M,\alpha)$ is hypertight, the contact product constructed above is hypertight as well as the Legendrian diagonal $\Delta_{\id}$. Thus one has

\begin{Thm}\label{thm: hype}
 If $(M,\xi)$ is hypertight then $(M,\xi)$ is strongly orderable. 
 \end{Thm}
 
This is equivalent to say that  any weakly non orderable (i.e. not strongly orderable) contact manifold has a contractible periodic Reeb orbit and thus satisfies the Weinstein conjecture. Note that the orderability of hypertight manifolds is already proved by Albers-Fuchs-Merry in \cite{AlMe} and Sandon in \cite{Marg_TH}.

Theorem \ref{thm:hypertight} was also recently announced by Sandon using a relative version of her Floer homology for translated points defined in \cite{Marg_TH}.

\subsection{A note about the hypotheses}
\label{rem_pardon}
The results here are not proven in the full generality that we would wish. The reason is that, for technical reasons, the construction of Cthulhu homology \cite{cthulhu} is currently restricted to \emph{symplectisation of contactisations} (or, more generally, symplectisations of hypertight contact manifolds). However, using recent work of Bao--Honda in \cite{Honda_Bao} or Pardon \cite{Pardon_CH}, it should be possible to define Legendrian contact homology as well as Cthulhu homology in a less restrictive setup. The authors believe that the correct restrictions are as follows:
\begin{enumerate}
\item The non-existence of a positive loop of one of the components of a link $\Lambda\sqcup \Lambda'$, such that the loop is contained in the complement of the other component, should also hold in a general contact manifold in the case when $LCC^*(\Lambda,\Lambda') \neq 0$ is nonzero. In other words, Theorem \ref{thm:liouville} and its corollaries should hold more generally. Observe that having a non-zero Legendrian contact homology in particular implies that the full contact homology of the ambient contact manifold also is nontrivial.
\item The non-existence of a contractible positive loop in Theorem \ref{thm:hypertight} is also expected to hold in a general contact manifold, under the assumption that the Legendrian submanifold $\Lambda$ satisfies the following property: Consider a $C^1$-small push-off $\Lambda'$ obtained as the one-jet $j^1f \subset \mathcal{J}^1\Lambda$ of a negative Morse function $f \colon \Lambda \to (-\epsilon,0)$ where the jet-space is identified with a standard contact neighbourhood of $\Lambda$. Then, we either require that the Reeb chord corresponding to the minimum of $f$ is a cycle inside $LCH_{\varepsilon,\varepsilon}^{n-1}(\Lambda,\Lambda')$ (which then defines a non-zero class, since it a priori is not a boundary) or, equivalently, that it is not a boundary inside $LCH^{\varepsilon,\varepsilon}_{n-1}(\Lambda,\Lambda')$ (which then again defines a non-zero class). This would provide a generalisation of Theorem \ref{thm:filling} to the non-fillable case; and
\item The long exact sequence produced by Theorem \ref{thm:les} should be possible to construct in the general case of a contact manifold whose full contact homology algebra is not acyclic. Again we must here make the assumption that the Legendrian submanifolds admit augmentations. The existence of this long exact sequence, along with its properties, is at the heart of the arguments for the results that we prove here.
\end{enumerate}

\subsection{Acknowledgements}
A part of this work was conducted during the symplectic program at the Mittag--Leffler institute in Stockholm, Sweden, the autumn of 2015. The authors would like to thank the institute for its hospitality and great research environment. This project was finalised while the third author was professeur invit\'e at Universit\'e de Nantes, and he would like to express his gratitude for a very pleasant stay.

\section{Floer theory for Lagrangian cobordisms}
\label{sec:cthulhu-compl-hypert}

In this section we recall the necessary background needed regarding the Floer homology complex constructed in \cite{cthulhu} by the first and third authors together with Ghiggini--Golovko. This is a version of Lagrangian intersection Floer homology defined for a pair consisting of two exact Lagrangian cobordisms in the symplectisation. In order to circumvent technical difficulties, we here restrict attention to the cases when either
\begin{itemize}
\item $(M,\alpha)$ is a contactisation $(P \times \R,\alpha_{\OP{std}})$ of a Liouville manifold endowed with its standard contact form, or
\item $(M,\alpha)$ as well as all Legendrian submanifolds considered are hypertight.
\end{itemize}
We refer to Section \ref{rem_pardon} for a discussion about these requirements, along with descriptions of less restrictive settings in which we believe our results hold.

\subsection{Generalities concerning Lagrangian cobordisms}
An \emph{exact Lagrangian cobordism $\Sigma \subset (\R \times M,d(e^t\alpha))$ from $\Lambda^- \subset (M,\xi)$ to $\Lambda^+ \subset (M,\xi)$} is a submanifold satisfying the following properties:
\begin{itemize}
\item $\Sigma\subset \R \times M$ is properly embedded and half-dimensional, and $e^t\alpha$ is exact when pulled back to $\Sigma$. (I.e.~$\Sigma$ is an exact Lagrangian submanifold.)
\item Outside of a subset of the form $(T_-,T_+) \times M$ for some numbers $T_- \le T_+$, the submanifold $\Sigma$ coincides with the cylinders $(-\infty,T_-] \times \Lambda_-$ (the so-called \emph{negative end}) and $[T_+,+\infty) \times \Lambda_+$ (the so-called \emph{positive end}), respectively. (The Lagrangian condition implies that $\Lambda^\pm \subset (M,\xi)$ are Legendrian submanifolds.)
\item There is a primitive of the pull-back of $e^t\alpha$ which is globally constant when restricted to either of the two cylindrical ends above. (When $\Lambda^\pm$ both are connected, this automatically holds.)
\end{itemize}
In the case when $\Lambda^-=\emptyset$ we say that $\Sigma$ is a \emph{filling}. A Lagrangian cobordism for which $L \cap [T_-,T_+] \times M$ is diffeomorphic to a cylinder is called a \emph{Lagrangian concordance}. Note that the exactness of the pull-back of $e^t\alpha$ is automatic in this case.

The exactness allows us to associate a potential $f_\Sigma \colon \Sigma \to \R$ defined uniquely by the requirements that it is the primitive of the pull-back of $e^t\alpha$ that vanishes on the negative end of $\Sigma$.

Given exact Lagrangian cobordisms $\Sigma^-$ from $\Lambda^-$ to $\Lambda$, and $\Sigma^+$ from $\Lambda$ to $\Lambda^+$, their \emph{concatenation} is the following exact Lagrangian cobordism. After a translation of the $\R$-coordinate, we may assume that $\Sigma^- \cap \{ t \ge -1\}$ and $\Sigma^+ \cap \{ t \le 1 \}$ both are trivial cylinders over $\Lambda$. The concatenation is then defined to be
$$ \Sigma^- \odot \Sigma^+ = (\Sigma^- \cap \{ t \le 0\}) \: \cup \: (\Sigma^+ \cap \{ t \ge 0 \}) \subset \R \times M,$$
which can be seen to be an exact Lagrangian cobordism from $\Lambda^-$ to $\Lambda^+$.

\subsection{Linearised Legendrian contact cohomology}
\label{sec:lch}

We start with a very brief recollection of the Chekanov--Eliashberg algebra, which is a differential graded algebra (DGA for short) $(\mathcal{A}(\Lambda),\partial)$ associated to a Legendrian submanifold $\Lambda \subset (M,\alpha)$ together with an auxiliary choice of cylindrical almost complex structure on the symplectisation $(\R \times M,d(e^t\alpha))$. The algebra is unital, fully non commutative, and freely generated by the set of Reeb chords $\mathcal{R}(\Lambda)$ on $\Lambda$ (which are assumed to be generic). Here we restrict attention to the case when the algebra is defined over the ground field $\Z_2$ of two elements. There is a grading defined by the Conley--Zehnder index, which we omit from the description. The differential is defined by counts of pseudoholomorphic discs in the symplectisation $(\R \times M,d(e^t\alpha))$ that are rigid up to translation. In the cases under consideration, the details can be found in \cite{LCHgeneral}. Also see \cite{LiftingPseudoholomorphic} for the relations between the version defined by counting pseudoholomorphic discs on the Lagrangian projection (which makes sense when $M=P\times\R$ is a contactisation) and the version defined by counting pseudoholomorphic discs in the symplectisation.

We will be working on the level of the so-called \emph{linearised} Legendrian contact cohomology complexes; these are complexes obtained from the Chekanov--Eliashberg DGA by Chekanov's linearisation procedure in \cite{Chekanov_DGA_Legendrian}. The latter complex has an underlying graded vector space with basis given by the Reeb chords, and the differential is associated to a so-called \emph{augmentation} of the DGA, which is a unital DGA morphism
$$ \varepsilon \colon (\mathcal{A}(\Lambda),\partial) \to \Z_2.$$
Observe that augmentations need not exist in general. Even if augmentations are purely algebraic objects, they are in many cases geometrically induced. For instance, an exact Lagrangian filling of $\Lambda$ gives rise to an augmentation; see \cite{RationalSFT} by Ekholm as well as \cite{Ekhoka} by Ekholm--Honda--K\'alm\'an.

Given augmentations $\varepsilon_i \colon \mathcal{A}(\Lambda_i)\to \Z_2$, $i=0,1$, the fact that the differential counts connected discs implies that there is an induced augmentation $\varepsilon$ of the Chekanov--Eliashberg algebra $\mathcal{A}(\Lambda_0 \cup \Lambda_1)$ of the disconnected Legendrian submanifold uniquely determined as follows: it restricts to $\varepsilon_i$ on the respective components while it vanishes on the chords between the two components. Using this augmentation, Chekanov's linearisation procedure can be used to produce a complex
\[(LCC_*^{\varepsilon_0,\varepsilon_1}(\Lambda_0,\Lambda_1),\partial^{\varepsilon_0,\varepsilon_1})\]
with underlying vector space having basis given by the Reeb chords $\mathcal{R}(\Lambda_1,\Lambda_0)$ from $\Lambda_1$ to $\Lambda_0$ (note the order!). We will instead mostly be working with the associated dual complex
\[(LCC^*_{\varepsilon_0,\varepsilon_1}(\Lambda_0,\Lambda_1),d_{\varepsilon_0,\varepsilon_1}),\]
called the \emph{linearised Legendrian contact cohomology complex}, with induced cohomology group $LCH^*_{\varepsilon_0,\varepsilon_1}(\Lambda_0,\Lambda_1)$. 

In the hypertight case, we restrict our attention to Reeb chords living in a fixed homotopy class $\alpha \in \pi_1(M,\Lambda)$,  which generate a subcomplex that we denote by
\[LCC^{\alpha,*}_{\varepsilon_0,\varepsilon_1}(\Lambda_0,\Lambda_1) \subset (LCC^*_{\varepsilon_0,\varepsilon_1}(\Lambda_0,\Lambda_1),d_{\varepsilon_0,\varepsilon_1}).\]
Recall that there is a canonical augmentation in this case which sends every generator to zero. From now on we will only use the trivial augmentation in the hypertight case, and therefore the aforementioned complex will simply be denoted by $LCC^{\alpha,*}(\Lambda_0,\Lambda_1)$.
\begin{Rem} Closed Legendrian submanifolds of a contactisation generically have a finite number of Reeb chords. However, in the hypertight case we cannot exclude the possibility of the existence of infinitely many Reeb chords in a given homotopy class. The latter complex thus has an underlying vector space which is a \emph{direct product}.
\end{Rem}
We proceed by giving some more details concerning the definition of the differential of the Legendrian contact cohomology complex for a pair of Legendrian submanifolds.

Use $\gamma^\pm \in \mathcal{R}(\Lambda_1,\Lambda_0)$ to denote Reeb chords from $\Lambda_1$ to $\Lambda_0$, and $\boldsymbol{\delta}=\delta_1\cdots \delta_{i-1}$, $\boldsymbol{\zeta}=\zeta_{i+1}\cdots\zeta_d$ to denote words of Reeb chords in $\mathcal{R}(\Lambda_0)$ and $\mathcal{R}(\Lambda_1)$, respectively. The differential is defined by the count
\begin{equation}
d_{\varepsilon_0,\varepsilon_0}(\gamma^-)=\sum\limits_{\boldsymbol{\delta},\boldsymbol{\zeta},\gamma^+}\#_2\mathcal{M}(\gamma^+;\boldsymbol{\delta},\gamma^-,\boldsymbol{\zeta})\varepsilon_0(\boldsymbol{\delta})\varepsilon_1(\boldsymbol{\zeta})\gamma^+.\label{eq:boundingcochain}
\end{equation}
of pseudoholomorphic discs in $\R \times M$ having boundary on $\R \times (\Lambda_0 \cup \Lambda_1)$ and strip-like ends, and which are rigid up to translation of the $\R$-factor. More precisely, the solutions inside $\mathcal{M}(\gamma^+;\boldsymbol{\delta},\gamma^-,\boldsymbol{\zeta})$ are required to have a positive puncture asymptotic to $\gamma^+$ at $t=+\infty$, and negative punctures asymptotic to $\gamma^-$, $\delta_j$, and $\zeta_j$, at $t=-\infty$. We refer to \cite{cthulhu} for more details on the definition of these moduli spaces. 

The \emph{length} of a Reeb chord is defined by the formula
$$\ell(\gamma):=\int_\gamma dz >0,\:\: \gamma \in \mathcal{R}(\Lambda).$$
The positivity of the so-called $d\alpha$-energy of the above pseudoholomorphic discs implies that the differential respects the filtration induced by the length of the Reeb chords in the following way: the coefficient of $\gamma^+$ above vanishes whenever $\ell(\gamma^+) \le \ell(\gamma^-)$. Also, see Lemma \ref{lem:action} below.

One can interpret augmentations as being bounding cochains (in the sense of \cite{fooo}) for Legendrians. Using this terminology, the differential defined by Equation \eqref{eq:boundingcochain} is induced by the choices  of bounding cochains $\varepsilon_i$ for $\Lambda_i$, $i=0,1$.

\subsection{The Cthulhu complex.}
\label{sec:cthulhu-complex}
We proceed to describe the construction of the Cthulhu complex from \cite{cthulhu} defined for a pair $(\Sigma_0,\Sigma_1)$ of exact Lagrangian cobordisms inside the symplectisation $(\R \times M,d(e^t\alpha))$. The starting point for this theory is the version of wrapped Floer cohomology defined \cite{RationalSFT2} by Ekholm using the analytic setup of symplectic field theory. Wrapped Floer homology is a version of Lagrangian intersection Floer homology for exact Lagrangian fillings. The theory in \cite{cthulhu} is a generalisation to the case when the negative end of the cobordism is not necessarily empty.

In order to deal with certain bubbling phenomena involving negative ends, one must require that the Legendrians at the negative ends admit augmentations. Again, augmentations will be used as bounding cochains. In the following we thus assume that we are given a pair $\Sigma_i$, $i=0,1$, of exact Lagrangian cobordisms from $\Lambda_i^-$ to $\Lambda_i^+$, together with choices of augmentations $\varepsilon_i$ of $\Lambda_i^-$. There are augmentations $\varepsilon_i^+=\varepsilon_i\circ\Phi_{\Sigma_i}$ of $\Lambda_i^+$ obtained as the pull-backs of the augmentations $\varepsilon_i^-$ under the unital DGA morphism induced by the cobordism $\Sigma_i$; see \cite{RationalSFT} by Ekholm as well as \cite{Ekhoka} for more details. Let $CF^*(\Sigma_0,\Sigma_1)$ be the graded $\Z_2$-vector space with basis given by the intersection points $\Sigma_0 \cap \Sigma_1$ (which all are assumed to be transverse). Again, we omit gradings from the discussion.

We are now ready to define the {\it Cthulhu complex}, which is the graded vector space
\begin{eqnarray*}
& & \Cth_*(\Sigma_0,\Sigma_1):=C^*_{+\infty} \oplus CF^*(\Sigma_0,\Sigma_1)\oplus C^{*-1}_{-\infty}, \\
& & C^*_{-\infty}(\Sigma_0,\Sigma_1):=LCC_{\varepsilon_0,\varepsilon_1}^*(\Lambda^-_0,\Lambda_1^-),\\
& & C^*_{+\infty}(\Sigma_0,\Sigma_1):=LCC_{\varepsilon_0^+,\varepsilon_1^+}^*(\Lambda^+_0,\Lambda_1^+).
\end{eqnarray*}
with differential of the form
$$\mathfrak{d}_{\varepsilon_0,\varepsilon_1}=\begin{pmatrix}
d_{++} & d_{+0} & d_{+-}\\
0 & d_{00} & d_{0-}\\
0 & d_{-0} & d_{--}
\end{pmatrix}.$$
The entries $d_{++}=d_{\varepsilon_0^+,\varepsilon_1^+}$ and $d_{--}=d_{\varepsilon_0,\varepsilon_1}$ are the linearised Legendrian contact cohomology differentials described in Section \ref{sec:lch}, while the rest of the entries are defined by augmented counts of pseudoholomorphic strips with boundary on $\Sigma_0$ and $\Sigma_1$ having appropriate asymptotics. See Figure \ref{fig:cthdiff} for a schematic picture of the strips involved. Recall that, typically, the strips also have additional negative asymptotics to Reeb chords on $\Lambda_i^-$, $i=0,1$, and that all counts are `weighted' by the values of the chosen augmentations on these chords (similarly to as in Formula \eqref{eq:boundingcochain}).

\begin{figure}[ht!]
\vspace{5mm}
\labellist
\pinlabel{$\mathbb{R}\times \Lambda^+_0$} [tr] at 0 515
\pinlabel{$\color{red}\mathbb{R}\times\Lambda^+_1$} [tl] at 75 515
\pinlabel{$d_{++}$} [tl] at 4 515
\pinlabel{$\Sigma_0$} [tr] at 295 545
\pinlabel{$\color{red}\Sigma_1$} [tl] at 380 545
\pinlabel{$d_{+0}$} [tl] at 304 545
\pinlabel{$\Sigma_0$} [tr] at 655 515
\pinlabel{$\color{red}\Sigma_1$} [tl] at 730 515
\pinlabel{$d_{+-}$} [tl] at 660 515
\pinlabel{$\Sigma_0$} [tr] at 298 325
\pinlabel{$\color{red}\Sigma_1$} [tl] at 375 325
\pinlabel{$d_{00}$} [tl] at 310 325
\pinlabel{$\Sigma_0$} [tr] at 670 325
\pinlabel{$\color{red}\Sigma_1$} [tl] at 725 325
\pinlabel{$d_{0-}$} [tl] at 670 305
\pinlabel{$\Sigma_0$} [tl] at 370 150
\pinlabel{$\color{red}\Sigma_1$} [tr] at 305 150
\pinlabel{$d_{-0}$} at 200 146
\pinlabel{$\mathbb{R}\times ({\color{red}\Lambda^-_1}\sqcup \Lambda_0^-)$} [tr] at 170 60
\pinlabel{$\Sigma_0$} [tr] at 660 105
\pinlabel{$\color{red}\Sigma_1$} [tl] at 730 105
\pinlabel{$d_{--}$} [tl] at 662 105
\pinlabel{$\text{\small{out}}$} at 40 585
\pinlabel{$\text{\small{in}}$} at 40 403
\pinlabel{$\text{\small{out}}$} at 336 585
\pinlabel{$\text{\small{in}}$} at 336 450
\pinlabel{$\text{\small{out}}$} at 336 400
\pinlabel{$\text{\small{in}}$} at 336 222
\pinlabel{$\text{\small{in}}$} at 336 180
\pinlabel{$\text{\small{out}}$} at  695 585
\pinlabel{$\text{\small{in}}$} at 695 403
\pinlabel{$\text{\small{out}}$} at  695 365
\pinlabel{$\text{\small{in}}$} at 695 223
\pinlabel{$\text{\small{out}}$} at  695 175
\pinlabel{$\text{\small{in}}$} at 695 -2
\pinlabel{$\text{\small{out}}$} at 200 100
\endlabellist
\centering
\includegraphics[width=8cm]{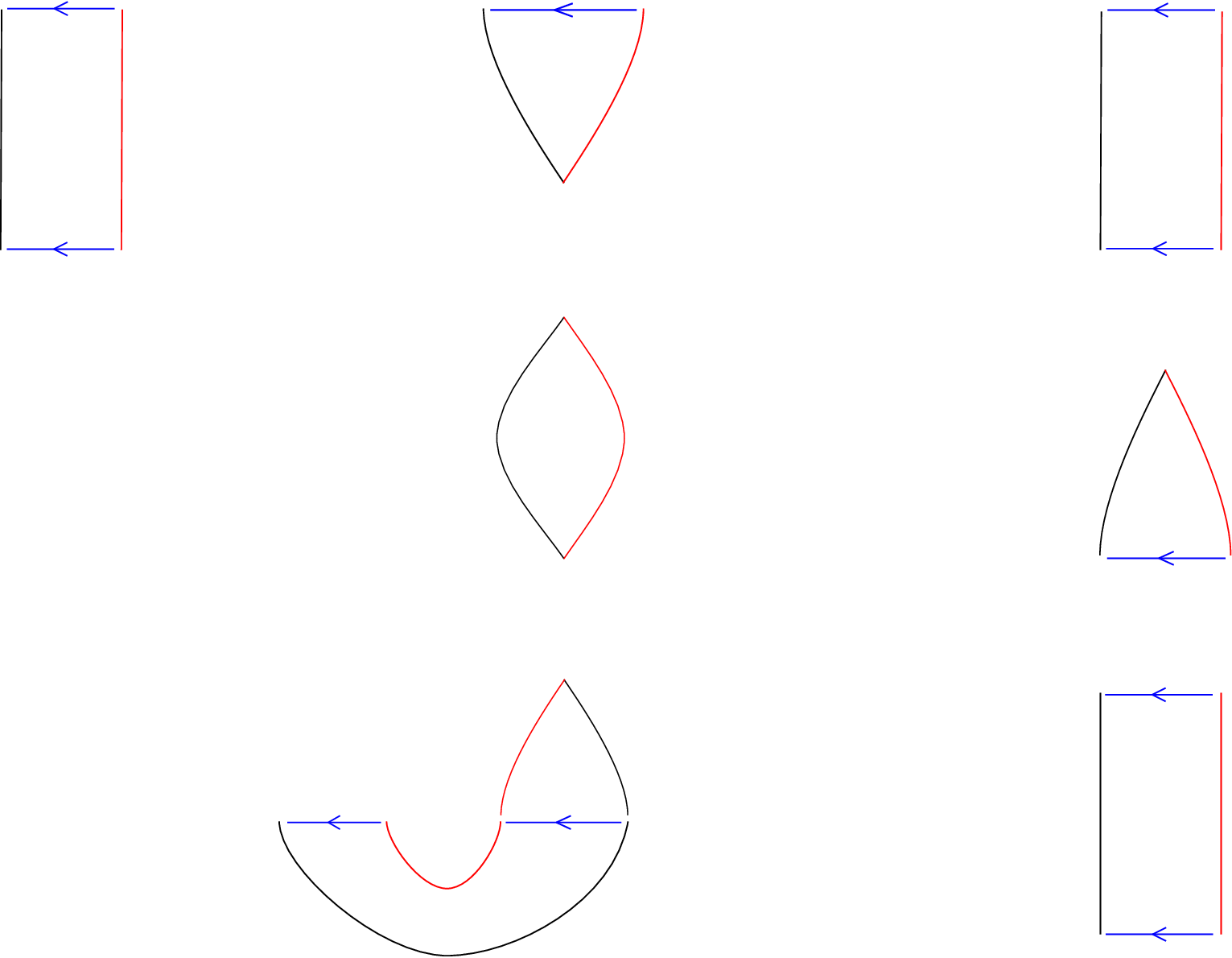}
\caption{Curves contributing to the Cthulhu differential; \emph{in} and \emph{out} denote the input and output of the respective component of the differential.}
\label{fig:cthdiff}
\end{figure}
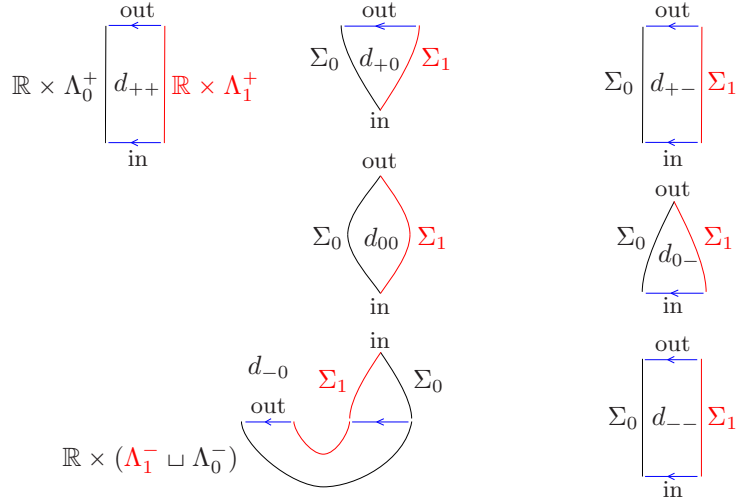

For us it will be important to consider the behaviour of $\mathfrak{d}_{\varepsilon_0,\varepsilon_1}$ with respect to a particular action filtration. To that end, for an intersection point $p \in \Sigma_0 \cap \Sigma_1$ we associate the action
\[ \mathfrak{a}(p):=f_{\Sigma_1}(p)-f_{\Sigma_0}(p) \in \R.\]
Assuming that $\Sigma_i$ are both cylindrical outside of $(T_-,T_+) \times M$, to a Reeb chord generator $\gamma^\pm \in \mathcal{R}(\Lambda_1^\pm,\Lambda_0^\pm)$ we can then associate the action
\[ \mathfrak{a}(\gamma^\pm):=e^{T_\pm}\ell(\gamma^\pm)+f_{\Sigma_1}^\pm-f_{\Sigma_0}^\pm \in \R,\]
where $f_{\Sigma_i}^\pm \in \R$ are defined to be the value of $f_{\Sigma_i}$ on the positive and negative cylindrical ends, respectively (thus $f_{\Sigma_i}^-=0$ by our conventions).
\begin{Lem}
\label{lem:action}
If the coefficient of $y$ is non-vanishing in the expression $\mathfrak{d}_{\varepsilon_0,\varepsilon_1}(x)$, then it follows that $\mathfrak{a}(y) > \mathfrak{a}(x)$.
\end{Lem}
We now state the well-definedness and invariance properties of the Cthulhu complex. For the precise requirements concerning the almost complex structure we refer to \cite{cthulhu}.

Assume that $(M,\alpha)$ is the contactisation of a Liouville manifold, and that $\Sigma_i \subset (\R \times M,d(e^t\alpha))$, $i=0,1,$ are exact Lagrangian cobordisms whose negative ends $\Lambda^-_i$ admit augmentations $\varepsilon_i$.
\begin{Thm}[Theorems 4.1, 6.6 in \cite{cthulhu}]
\label{thm:invariance0}
For an appropriate choice of an almost complex structure on the symplectisation, $\mathfrak{d}_{\varepsilon_0,\varepsilon_1}^2=0$ and the induced complex $\Cth_*(\Sigma_0,\Sigma_1)$ is acyclic.
\end{Thm}

\begin{Thm}[Proposition 6.4 in \cite{cthulhu}]
\label{thm:invariance1}
If the cobordisms $\Sigma_i$, $i=0,1$, are compactly supported Hamiltonian isotopic to $\Sigma_i'$, then there is an induced quasi-isomorphism
\begin{gather*} \phi \colon \Cth_*(\Sigma_0,\Sigma_1) \to \Cth_*(\Sigma_0',\Sigma_1'),\\
\phi=\begin{pmatrix}
\phi_+ & * & * \\
0 & * & * \\
0 & \phi_{-0} & \id_{C_{-\infty}(\Lambda_0^-,\Lambda_1^-)}
\end{pmatrix},
\end{gather*}
where $\phi_{-0}$ vanishes in the case when there are no Reeb chords from $\Lambda_0$ to $\Lambda_1$, and:
\begin{enumerate}
\item The component
$$\phi_+ \colon LCC_{\varepsilon_0 \circ \Phi_{\Sigma_0},\varepsilon_1 \circ \Phi_{\Sigma_1}}^*(\Lambda_0^+,\Lambda_1^+) \to LCC_{\varepsilon_0 \circ \Phi_{\Sigma_0'},\varepsilon_1 \circ \Phi_{\Sigma_1'}}^*(\Lambda_0^+,\Lambda_1^+) $$
is an isomorphism of complexes which for generators $a,b$ satisfies\\ $\langle \phi_+(a),a \rangle =1$, while $\langle \phi_+(a),b \rangle=0$ holds whenever $\ell(a)>\ell(b)$.
\item Consider the subspaces $\Cth_*^{[\mathfrak{a}_0,+\infty)}(\Sigma_0^s,\Sigma_1^s) \subset \Cth_*(\Sigma_0^s,\Sigma_1^s)$ consisting of the generators of action at least $\mathfrak{a}_0 \in \R$ where $(\Sigma_0^s,\Sigma_1^s)$ is the Hamiltonian isotopy of the pair of cobordisms. Under the additional geometric assumption that a neighbourhood of these generators, as well as their actions, are fixed during the entire isotopy, it follows that
$$\phi(\Cth_*^{[\mathfrak{a}_0,+\infty)}(\Sigma_0,\Sigma_1)) \subset \Cth_*^{[\mathfrak{a}_0,+\infty)}(\Sigma_0',\Sigma_1')$$
holds as well.
\end{enumerate}
\end{Thm}
\begin{proof}
The claims not contained in the formulations of the referred result in \cite{cthulhu} are the following:
\begin{enumerate}[label=(\roman*)]
\item The condition for the vanishing of $\phi_{-0}$;
\item The statement in Part (1) which claims that the chain isomorphism $\phi_+$ is `upper triangular' with respect to the action filtration; and
\item The action preserving properties stated in Part (2).
\end{enumerate}

The statements follow from studying the proof of the invariance result \cite[Proposition 6.4]{cthulhu} which is shown using bifurcation analysis. Here there is one caveat: it is necessary first to apply Proposition 8.2 in the same article. This is done in order to interchange the Reeb chord generators of $C^*_{\pm\infty}(\Sigma_0,\Sigma_1)$ with intersection points by geometrically `wrapping' the ends, e.g.~by an application of the Hamiltonian isotopy $\phi^s_{(\beta-1)\partial_z}$ to the component $\Sigma_0$ (see Section \ref{sec:wrapping}). According to the aforementioned result, we may assume that the isomorphism class of the complex is left unchanged under such a modification. Alternatively, one may also argue as the invariance proof \cite[Section 4.2.1]{RationalSFT2}, which is based upon abstract perturbations (and which stays in the more symmetric setup of SFT).

The bifurcation analysis roughly works as follows. A generic Hamiltonian isotopy produces a finite number of `handle-slides' and `birth/deaths' on the geometric side. On the algebraic side handle-slides and birth/deaths then correspond to chain isomorphisms of the form $x \mapsto x + K(x)$ (defined on each generator) and stabilisations by an acyclic complex (up to a chain isomorphism), respectively.

In this case we are only concerned with generators which cannot undergo any birth/death moves, and we can therefore ignore them. What suffices is thus to check the action properties for each chain isomorphism induced by a handle-slide. Recall that the term $K(x)$ in such a chain isomorphism is defined by a count of pseudoholomorphic strips of expected dimension -1;  such strips generically exist inside a one-parameter family, where they appear as rigid solutions.

Claim (i): This follows by a neck-stretching argument, since a pseudoholomorphic disc contributing to a nonzero term $K(x)$ in the definition of $\phi_{-0}$ would break into a configuration involving a Reeb chord from $\Lambda_0$ to $\Lambda_1$. This configuration is similar to the one which in the definition of the term $d_{-0}$ of the differential shown in Figure \ref{fig:cthdiff}.

Claims (ii) and (iii): The claims follow by the same reason as to why the differential is action increasing, i.e.~since non-constant pseudoholomorphic discs are of positive energy.

\end{proof}
It will also be useful to formulate the following refined invariance properties, which applies under certain additional assumptions on the cobordisms.
\begin{Cor}[\cite{cthulhu}]
\label{cor:invariancerefined}
In the above setting, we make the additional assumption that there are no Reeb chords starting on $\Lambda_0^-$ and ending on $\Lambda_1^-$, and that each cobordism $\Sigma_i$ is compactly supported Hamiltonian isotopic to cobordisms $\Sigma_i'$ satisfying $\Sigma_0' \cap \Sigma_1'=\emptyset$. Then $d_{-0}=0$ holds for both complexes, and the quasi-isomorphism $\phi$ in Theorem \ref{thm:invariance1} can be assumed to map all intersection points into the subcomplex $C_{+\infty}^*(\Sigma_0',\Sigma_1') \subset \Cth_*(\Sigma_0',\Sigma_1')$ (i.e. $\phi_{-0}=0$).
\end{Cor}

\subsection{The Cthulhu complex in the hypertight case}
\label{sec:ht}

We are also interested in the (very) special situation when the negative ends $\Lambda^-_i \subset (M,\alpha),$ $i=0,1,$ of both cobordisms are hypertight Legendrian submanifolds of a closed hypertight contact manifold.

At a first glance, this is a slightly more general setting than that considered in \cite{cthulhu}. However having hypertight Legendrian ends implies no additional difficulties regarding transversality results for holomorphic curves. Indeed, since we consider the canonical augmentation only, every holomorphic curve involved in the definition of the Cthulhu complex will be an honest strip, and thus in particular have at least one mixed Reeb chord among its asymptotics. For that reason, we can apply \cite[Proposition 3.2]{cthulhu} to achieve transversality (in particular this proposition makes no use of the special type of almost complex structure used on the symplectisation of a contactisation). As already mentioned in \cite[Section 1.4.1]{cthulhu} the other reason contactisations were used in \cite{cthulhu} was to get acyclicity of the complex. This property we will here deduce in the special case considered, using the assumption that the cobordisms are either trivial cylinders or traces of contractible loops.

Our definition of hypertight Legendrian leaves us with  the following caveat (recall that our notion of hypertightness does not require the Reeb chords, or orbits, to be non-degenerate):

In the case when the Legendrian is hypertight, we do not know if it is possible to make a small perturbation $\Lambda_1$ of $\Lambda_0$ that simultaneously satisfies the properties that:
\begin{enumerate}
\item each of $\Lambda_i$ is hypertight and has non-degenerate Reeb chords,
\item the contact manifold is hypertight with non-degenerate Reeb orbits, and
\item the Reeb chords between $\Lambda_0$ and $\Lambda_1$ are non-degenerate.
\end{enumerate}
However, given any choice of $L > 0$, it is possible to ensure that the above properties are satisfied for all the chords of length \emph{at most} $L$.

For that reason, we must use the following modified versions of the complexes. We tacitly restrict attention to the generators below some fixed, but sufficiently large, action $L \gg 0$ when considering the different complexes. That this indeed makes sense follows from the action properties satisfied by the differential and the maps appearing in the invariance statements.

Obviously, working below a fixed action is in general not possible to combine with a full invariance. However, when interested in the invariance under a \emph{fixed} deformation, we can always adjust the parameter $L>0$ in order for this invariance to hold. Here it is important to note that we only consider invariance properties under \emph{compactly supported} Hamiltonian isotopies of a pair of exact Lagrangian cobordisms. For a fixed such deformation, all generators concerned that can undergo a deformation thus satisfy an a priori action bound (depending only on the involved cobordisms together with the fixed Hamiltonian isotopy).

In this case the Chekanov--Eliashberg algebra generated by the \emph{contractible} chords has a canonical augmentation that sends every generator to zero.  Restricting attention to only those chords being of action smaller than $L>0$,  the induced differential $\mathfrak{d}_L$ is defined by counting honest pseudoholomorphic strips, i.e.~strips without additional boundary punctures asymptotic to Reeb chords of $\Lambda^-$. This situation is very similar to that in the paper \cite{ElHoSa}, in which both Lagrangians also are non-compact.

\begin{Thm}[\cite{cthulhu}]
For any two exact Lagrangian cobordisms $\Sigma_i \subset (\R \times M,d(e^t\alpha))$ having hypertight negative ends in a hypertight contact manifold, we have $\mathfrak{d}_L^2=0$.
\end{Thm}

In the hypertight case, we only formulate the invariance theorem for the special case that is needed in our proofs. Assume that the two Legendrian submanifolds $\Lambda_i \subset (M,\alpha)$, $i=0,1$, are hypertight, where $\Lambda_0$ moreover is obtained from $\Lambda_1$ by the time-$Z$ Reeb flow, $Z > 0$, followed by a generic $C^1$-small Legendrian perturbation.

We begin with the case when the two exact Lagrangian cobordisms $\Sigma_i \subset (\R \times M,d(e^t\alpha))$ have been obtained from $\R \times \Lambda_i$ by a compactly supported Hamiltonian isotopy. We consider the Cthulhu complex generated by only those Reeb chords and intersection points living in the component
$$0 \in \pi_0(\Pi(\R \times M;\Sigma_0,\Sigma_1))$$
of paths from $\Sigma_1$ to $\Sigma_0$ in $\R \times M$ containing the (perturbations of the) Reeb chords from $\{T_-\} \times \Lambda_1$ to $\{T_-\} \times \Lambda_0$ of length precisely equal to $Z>0$. Since the differential counts honest strips (i.e.~with no additional boundary punctures), it is clear that this defines a subcomplex that will be denoted by $\Cth^0_*(\Sigma_0,\Sigma_1)$. 
\begin{Thm}
\label{thm:invarianceht}
Under the above assumptions, the complex $\Cth^0_*(\Sigma_0,\Sigma_1)$ is acyclic and satisfies $d_{-0}=0$. Moreover, there is a homotopy equivalence
$$ \phi \colon \Cth^0_*(\Sigma_0,\Sigma_1) \to \Cth^0_*(\R \times \Lambda_0,\R \times \Lambda_1)$$
of acyclic complexes such that:
\begin{enumerate}
\item Its restriction to the subcomplex $C_{+\infty}^*(\Sigma_0,\Sigma_1)$ is the identity map
$$ C_{+\infty}^*(\Sigma_0,\Sigma_1) \xrightarrow{=} C_{+\infty}^*(\R \times \Lambda_0,\R \times \Lambda_1)$$
of complexes having the same set of generators; and
\item All intersection points are mapped into the subcomplex $$C_{+\infty}^*(\R \times \Lambda_0,\R \times \Lambda_1) \subset \Cth^0_*(\R \times \Lambda_0,\R \times \Lambda_1).$$
\end{enumerate}
\end{Thm}
\begin{proof}
The proof is similar to that of Corollary \ref{cor:invariancerefined}. Observe that, since the cylindrical ends are fixed during the isotopy, it makes sense to restrict attention to the specified homotopy class of intersection points whilst performing the bifurcation analysis.

We have $d_{-0}=0$ when restricted to the generators in the homotopy class under consideration. Here the crucial point is that, by hypertightness together with the choice of push-off, there are no Reeb chords from $\Lambda_0$ to $\Lambda_1$ that end up in the homotopy class $0 \in \pi_0(\Pi(\R \times M;\Sigma_0,\Sigma_1))$ when parametrised by \emph{backwards time}, i.e.~using the negative Reeb flow. The rest follows as in the aforementioned proof. 

Note that the restriction of the homotopy equivalence to $C_{+\infty}^*(\Sigma_0,\Sigma_1)$ is the identity morphism, as opposed to the more general isomorphism of complexes in Part (1) of Theorem \ref{thm:invariance1}. This follows by topological reasons together with the fact that we are using the trivial augmentation; the latter obviously pulls back to the trivial augmentation under the DGA morphisms induced by the cobordisms $\Sigma_i$ under consideration (and their deformations by a compactly supported Hamiltonian isotopy).
\end{proof}

The following invariance holds in the more general situation when $\Sigma_i$ both are \emph{invertible} Lagrangian cobordisms; by this, we mean that there are cobordisms $U_i,V_i$ for which the concatenations $\Sigma_i \odot U_i$ as well as $U_i \odot V_i$ both can be performed, and such that the resulting exact Lagrangian cobordisms all are compactly Hamiltonian isotopic to trivial cylinders. See \cite[Section 5.3]{LiftingPseudoholomorphic} for the basic properties of invertible Lagrangian cobordisms.
\begin{Rem}
\label{rem:htpyclass}
If $\Sigma_i$ are Lagrangian cylinders from $\Lambda_i$ to itself that are not compactly supported Hamiltonian isotopic to trivial cylinders, a path inside the slice $\{ T_- \} \times M$ living in the component $0 \in \pi_0(\Pi(\R \times M;\Sigma_0,\Sigma_1))$ specified above (e.g.~a Reeb chord from $\Lambda_1$ to $\Lambda_0$) may end up in a \emph{different} homotopy class when placed inside the slice $\{ T_+ \} \times M$.
\end{Rem}
\begin{Thm}[Theorem 5.7 in \cite{LiftingPseudoholomorphic}]
\label{thm:invariance2}
If $\Sigma_i$ are invertible exact Lagrangian cobordisms between hypertight Legendrians, then $\Cth_*^0(\Sigma_0,\Sigma_1)$ is an acyclic complex.
\end{Thm}
\begin{proof}
Since the negative ends are kept fixed during the Hamiltonian isotopy considered, the statement follows from the same proof as in the case of a pair of fillings (i.e.~in the case when the negative ends are empty). See e.g.~the proof of \cite[Theorem 5.7]{LiftingPseudoholomorphic}, which is based upon \cite[Section 4.2]{RationalSFT2}.
\end{proof}

In the situation where the Legendrian is the diagonal in the contact product of an hypertight contact manifold, we apply the theory from Zena\"{i}di's work \cite{Zenaidi} in the following way. One can consider a contact form of the type $\widehat{\alpha} =f_1 \alpha_1 +f_2 \alpha_2$ on the contact product, where
 $f_1,f_2 :\R \to \R$ are linear away from a compact set $[-N,N]$ and $\delta(t)=f_1'(t)f_2(t)-f_2'(t)f_1(t) \neq 0$ for all $t\in \R$ (this is the contact condition).
 The associated Reeb vector field is $R_{\widehat{\alpha}}=(-\frac{f_2'}{\delta} R,\frac{f_1'}{\delta} R,0)$. Note that periodic orbits of $R_{\widehat{\alpha}}$ are given as curves on the product of orbits of $R$, and thus the form is still hypertight. A chord of the Legendrian diagonal is given by a path $(\gamma_1(t),\gamma_2(t),0)$ such that $\gamma_1\star\gamma_2^{-1}$ is a periodic orbit of $R$, thus $\Delta$ is also relatively hypertight. Since  $R_{\widehat{\alpha}}$ has vanishing $\partial_t$-component, all chords are confined in $M\times M\times \{0\}$. By an argument based upon the maximum principle (see Theorem 5.3.9 in \cite{Zenaidi}) it follows that all holomorphic curves asymptotic to periodic orbits stay in the symplectisation of the compact region $M\times M\times [-\varepsilon,\varepsilon]$. For this reason, the Floer theory for cobordisms considered here can also be extended to the non-compact settings of the contact product.




\subsection{Wrapping}
\label{sec:wrapping}

In the Hamiltonian formulation of wrapped Floer cohomology the Reeb chord generators are exchanged for Hamiltonian chords arising when `wrapping' the ends of one of the Lagrangians. It will sometimes be necessary for us to perform such a wrapping as well, and for that reason we need to introduce the following Hamiltonian vector fields.

We start by observing the general fact that the isotopy $\phi^s_{g(t)R}$ generated by a vector field of the form $g(t)R \in T(\R \times M)$ generates a Hamiltonian isotopy, where $R$ denotes the Reeb vector field and $g \colon \R \to \R$ is an arbitrary smooth function.

Now consider the function $\beta \colon \R \to \R_{ \ge 0}$ shown in Figure \ref{fig:bulge}, which satisfies the following properties:
\begin{itemize}
\item $\beta(t)\equiv 0$ for all $t \notin [-\delta,T+\delta]$,
\item $\beta(t)\equiv 1$ for all $t \in [0,T]$, and
\item $\beta'(t)>0$ and $\beta'(t)<0$ holds for $t \in (-\delta,0)$ and $t \in (T,T+\delta)$, respectively.
\end{itemize}
The induced Hamiltonian vector field $\phi^s_{(\beta-1)R}$ can now be seen to wrap the ends of $\R \times M$ by the negative Reeb flow.

\begin{figure}[htp]
\vspace{1em}
\centering
\labellist
\pinlabel $\beta(t)$ at 155 76
\pinlabel $t$ at 315 11
\pinlabel $-\delta$ at 14 0
\pinlabel $1$ at 55 75
\pinlabel $T$ at 260 0
\pinlabel $T+\delta$ at 293 0
\endlabellist
\includegraphics[scale=0.8]{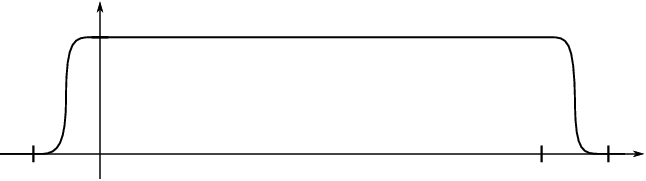}
\caption{The compactly supported bump-function $\beta(t)$.}
\label{fig:bulge}
\end{figure}
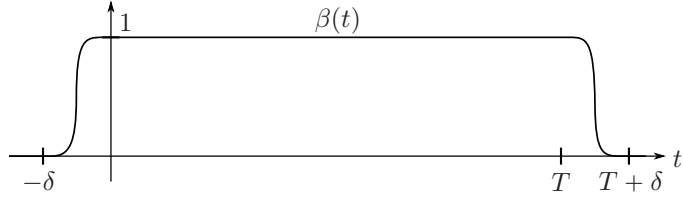

In addition, denote by $\beta^+(t)$ the function which is constantly equal to $\beta^+(t) \equiv 1$ when $t \le T$ and which coincides with $\beta(t)$ on $[T,+\infty)$; and $\beta^-(t)$ to denote the function which is constantly equal to $\beta^-(t) \equiv 1$ when $t \ge 0$ and which coincides with $\beta(t)$ on $(-\infty,0)$.

\subsection{Stretching the neck}
\label{sec:stretching}

Consider a (possibly disconnected) hypersurface
$$\{t_1,\hdots,t_n\} \times M \subset (\R \times M,d(e^t\alpha))$$
of contact type near which both Lagrangian cobordisms $\Sigma_i$, $i=0,1$, are cylindrical. As described in \cite[Section 3.4]{Bourgeois_&_Compactness}, we can stretch the neck along this hypersurface by considering an appropriate sequence $J_\tau$, $\tau \ge 0$, of compatible almost complex structures and then applying the SFT compactness theorem \cite{Bourgeois_&_Compactness}. More precisely, in neighbourhoods $[t_i-\epsilon,t_i+\epsilon] \times M$ the almost complex structure $J_\tau$ is induced by pulling back a fixed choice of cylindrical almost complex structure on $([t_i-\epsilon-\tau,t_i+\epsilon+\tau] \times M,d(e^t\alpha))$ by a diffeomorphism induced by an identification $[t_i-\epsilon,t_i+\epsilon] \cong [t_i-\epsilon-\tau,t_i+\epsilon+\tau]$. Also, see \cite[Section 5.1]{cthulhu} for a description in the setting considered here.

When stretching the neck by taking $\tau \to +\infty$, a sequence of $J_\tau$-holomorphic curves in $\R \times M$ with boundary on $\Sigma_0 \cup \Sigma_1$ have subsequences converging to so-called pseudoholomorphic buildings; we refer to \cite{Bourgeois_&_Compactness} and \cite{Abbas_book} for its definition. Roughly speaking, a pseudoholomorphic building consists of many levels containing pseudoholomorphic curves with boundary on the (completed) Lagrangian cobordisms contained between $(t_i,t_{i+1}) \times M$ for $i=0,\hdots,t_n$, and where $t_0=-\infty$, $t_{n+1}=+\infty$. See Figure \ref{fig:stretched} below for an example in the case when the building has three levels.

In order to obtain a bijection between rigid configurations before the limit and rigid pseudoholomorphic buildings one must assume that all involved components are transversely cut out, and then perform pseudoholomorphic gluing.

\section{Positive isotopies and Lagrangian concordance}
\label{sec:posit-isot-lagr}
The starting point of our analysis is the construction of a Lagrangian concordance from a Legendrian isotopy which should be though of as the `trace' of the isotopy. We choose to follow the construction of Eliashberg and Gromov \cite{Eliashberg_&_Lagrangian_Intersection_Finite}. Given a Legendrian isotopy
\begin{gather*}
\phi \colon [0,1] \times \Lambda \hookrightarrow (M,\xi),\\
\Lambda^s := \phi(s,\Lambda),
\end{gather*}
from $\Lambda^0$ to $\Lambda^1$, they produce an associated Lagrangian concordance
$$\Sigma_{\{\Lambda^s\}} \subset (\R \times M,d(e^t\alpha))$$
from $\Lambda^0$ to $\Lambda^1$ inside the symplectisation.

\begin{Rem}
Lagrangian concordances can be constructed out of a Legendrian isotopy in several different ways; see e.g.~\cite{chantraine_conc} or \cite{RationalSFT} for alternatives to the construction presented here. Note that, the primitive of $e^t\alpha$ for the construction from \cite{chantraine_conc} has the same value on the negative and positive ends. For the construction in \cite{RationalSFT}, the symplectic and Liouville structure are changed, making action considerations more delicate. The advantage of the techniques in \cite{Eliashberg_&_Lagrangian_Intersection_Finite} that we here choose to adapt is that, in the case of a positive (resp. negative) isotopy, the obtained cobordism has a primitive with value at the positive end being \emph{strictly greater} (resp. \emph{strictly smaller}) than the negative end. This property will turn out to be crucial.
\end{Rem}

Using a standard construction (for instance described in \cite{Colin_Sandon_Osci}), deforming the above Legendrian isotopy while fixing the endpoints, we can assume that the Legendrian isotopy $\{\Lambda^s\}$ has an associated \emph{contact Hamiltonian}
\begin{gather*}
H \colon [0,1] \times \Lambda \to \R,\\
H(s,q)=\alpha\left(\left.\frac{d}{dt}\right|_{t=s}\Lambda^t(q)\right),
\end{gather*}
that satisfies the following properties. There exists a decomposition $0=s_0<s_1<\cdots<s_k=1$ of the interval, and a number $\delta>0$, such that for all $i=0,\dots, k-1$:
\begin{itemize}
\item $H(s,q)|_{(s_i-\delta,s_i+\delta)}=\rho_i(s)$ (i.e. $H$ does not depend on $q$ near $s_i$);
\item $H(s,q)|_{(s_i,s_{i+1})}\not= 0$ (i.e. the isotopy is either positive or negative); and
\item $H(s_i,q) =0$.
\end{itemize}
The third condition enables us to extend the isotopy to be constant in the time intervals $s<0$ and $s>1$, while the contact Hamiltonian $H$ smoothly extends to zero for these times. Such an isotopy will be called a {\it zig-zag} isotopy.

To construct the sought concordance it suffices to perform the construction for the isotopy $\Lambda^s$ restricted to each interval $[s_i,s_{i+1}]$. The resulting concordances can then be stacked together by repeated concatenations in order to produce the sought Lagrangian concordance from $\Lambda^0$ to $\Lambda^1$.

\subsection{The concordance in the definite case}
In view of the above, we now restrict our attention to an isotopy $\{\Lambda^s\}_{s\in[0,1]}$ for which the contact Hamiltonian $H(s,q)$ satisfies $H(s,q)=\rho_0(s)$ for $s<\delta$, $H(s,q)=\rho_1(s)$ for $s>1-\delta$, $\rho_0(0)=\rho_1(1)=0$, and $|H(s,q)|\not=0$ for all $s\in (0,1)$. For all $\epsilon<\frac{1}{4}$ we choose a function $\chi_{\varepsilon}:\mathbb{R}\rightarrow\mathbb{R}$ such that
\begin{enumerate}
\item $\chi_{\epsilon}(s)=0$ for $s<0$ and $s>1$;
\item $\chi_\epsilon(s)=1$ for $s\in [2\epsilon,1-2\epsilon]$;
\item $\chi_\epsilon(s)=|\rho_0(s)|$ for $s<\min\{\delta,\epsilon\}$; and
\item $\chi_\epsilon(s)=|\rho_1(s)|$ for $s>\max\{1-\epsilon,1-\delta\}$.
\end{enumerate}

Given numbers $T>0$ and $\epsilon>0$ we are now ready to define the smooth map
\begin{gather}
C^{T,\varepsilon}\colon \mathbb{R}\times\Lambda \rightarrow \mathbb{R}\times M \nonumber\\
(s,q) \mapsto \begin{cases}
(Ts+\ln \frac{\chi_\epsilon(s)}{|H(s,q)|},\Lambda_s(q)), & \text{ for } s\in [0,1],\\
(Ts,\Lambda_0(q)), & \text{ for }s<0,\\
(Ts,\Lambda_1(q)), & \text{ for }s>1.
\end{cases}
\label{eq:7}
\end{gather}
Note that, from the definition of $\chi$, it follows that $C^{T,\varepsilon}$ is a smooth map. Namely, the function $\frac{\chi(s)_\epsilon}{H(s,q)}$ approaches $1$ as $s$ approaches either $0$ or $1$.

The following theorem is from \cite[Lemma 4.2.5]{Eliashberg_&_Lagrangian_Intersection_Finite}.
\begin{Prop}\label{prp:pos_cyl}
For all $T,\varepsilon > 0$ the map $C^{T,\varepsilon}$ is an immersed Lagrangian concordance. For $T$ sufficiently large, $C^{T,\varepsilon}$ is an embedded Lagrangian concordance from $\Lambda^0=\Lambda$ to $\Lambda^1$. Moreover, the pull back $(C^{T,\varepsilon})^*(e^t\alpha)$ is exact with primitive given by
$$f_{C^{T,\varepsilon}}(s,q)=\OP{sign}(H)\int_{-1}^s\chi_\varepsilon(t)e^{Tt}dt,$$
vanishing at $-\infty$. In particular, for a positive (resp. negative) isotopy, the primitive is positive (resp. negative) when restricted to the positive cylindrical end.
\end{Prop}
\begin{proof}
A simple computation shows that
$$(C^{T,\varepsilon})^*(e^t\alpha)=\frac{H}{|H|}\chi_\varepsilon(s)e^{Ts} ds$$
which implies the first and last statements. 

For the second statement, we note first that, since the isotopy is transverse to $\xi$ for all $s\in [0,1]$, there always exists a small $\delta>0$ such that the projection of $C^{T,\varepsilon}|_{(s-\delta,s+\delta)\times\Lambda}$ to $M$ is an embedding. This imply that every pair of points $(s_1,q_1)$ and $(s_2,q_2)$ (with $s_2>s_1$) such that $C(s_1,q_1)=C(s_2,q_2)$ satisfies
\begin{align}
\label{eq:10}
s_2-s_1&>2\delta,\\
T(s_2-s_1)&=\ln \frac{H(s_2)}{H(s_1)}-\ln\frac{\chi_\epsilon(s_2)}{\chi_\epsilon(s_1)}.
\end{align}
After choosing
$$T>\frac{1}{2\delta}\ln{\frac{\max{H}}{\min{H}}},$$
it follows that no such double point can exist.
\end{proof}

\begin{Rem}\label{rem:triv0T}
After increasing $T \gg 0$ even further, it is the case that
\begin{gather*}
C^{T,\epsilon}(\mathbb{R}\times \Lambda)\cap ((-\infty,0)\times M)=(-\infty,0)\times \Lambda^0, \\
C^{T,\epsilon}(\mathbb{R}\times \Lambda)\cap ((T,\infty)\times M)=(T,\infty)\times \Lambda^1,
\end{gather*}
are satisfied. We assume that such choices are made when using these concordances.
\end{Rem}

\subsection{The general case}
Given a general Legendrian isotopy $\{\Lambda^s\}$, we homotope it to a zig-zag isotopy. Concatenating the pieces of concordances produced Proposition \ref{prp:pos_cyl} above, we thus obtain the sought Lagrangian concordance
$$\Sigma_{\{\Lambda^s\}} \subset (\R \times M,d(e^t\alpha))$$
from $\Lambda^0$ to $\Lambda^1$.
\begin{Prop}
\label{prp:contractible}
For a contractible Legendrian loop $\{\Lambda^s\}$ with $\Lambda=\Lambda^0$, the Lagrangian concordance constructed above is compactly supported Hamiltonian isotopic to $\R \times \Lambda$.
\end{Prop}
\begin{proof}
The construction of these cylinders and of the homotopy making the Legendrian isotopy a zig-zag isotopy as in Lemma 2.1 of \cite{Colin_Sandon_Osci} is parametric. Thus, a homotopy of Legendrian isotopies will lead to an isotopy of Lagrangian concordances fixed outside of a compact subset. Since concordances are exact Lagrangians,  we can use the standard fact that an exact Lagrangian isotopy is generated by a Hamiltonian in order to obtain our sought isotopy. (The noncompactness causes no issue since the exact Lagrangian isotopy is compactly supported and since, at least near the Lagrangian, the Hamiltonian is locally constant outside of the support by its construction.)

More precisely, if $\{\Lambda^s\}$ is a contractible loop, the parametric construction gives a homotopy of zig-zag Legendrian isotopies from a small deformation of $\{\Lambda^s\}$ to a small deformation of the constant isotopy $\{\Lambda^0\}$. The latter deformation is a concatenation of isotopies of the form  $\{\phi_{\varepsilon\chi(t)}(\Lambda^0)\}$ where $\phi$ is the Reeb flow and $\chi$ is a bump function similar to the one considered above.

By construction the resulting Lagrangian concordance associated to $\{\Lambda^s\}$ is isotopic through exact Lagrangian embeddings, and by a standard result hence Hamiltonian isotopic, to the cylinder associated to to the cylinder $C'(t,q)$ that is the concatenation of the `graphs' $\{ (t,q); \: q \in \phi_{\varepsilon\chi(t)}(\Lambda^0)\}$. The latter cylinder can finally be explicitly seen to be Hamiltonian isotopic to the trivial cylinder $\mathbb{R}\times \Lambda^0$ by taking $\varepsilon \to 0$ in each piece simultaneously.
\end{proof}

\section{The Floer homology of the trace of a positive loop}\label{sec:floer-homology-trace}
Consider the Lagrangian concordance $\Sigma_0:=\Sigma_{\{\Lambda^s\}}$ obtained from the trace of a positive loop as constructed in Section \ref{sec:posit-isot-lagr}. We assume that the starting point of the loop is the Legendrian $\Lambda^0=\Lambda_0$, and we let $\Sigma_1:=\R \times \Lambda_1$ be the trivial cylinder over the Legendrian $\Lambda_1$. In particular, each of $\Sigma_i$, $i=0,1$, is an exact Lagrangian concordance from $\Lambda_i$ to itself.

In this section we compute the Floer homology complex $\Cth_*(\Sigma_0,\Sigma_1)$ for this pair of cobordisms, as defined in Section \ref{sec:cthulhu-compl-hypert}. We start by proving some results of a more technical nature concerning these complexes, which later will be used when showing the nonexistence of positive loops. For all results the assumption that $\Lambda_i$ has an augmentation is crucial. Denote by $\varepsilon_i$ this augmentation, and by $\varepsilon^+_i:=\varepsilon_i\circ\Phi_{\Sigma_i}$ its pull-back under the DGA morphism associated to $\Sigma_i$. The differential of the complex $\mathfrak{d}_{\varepsilon_0,\varepsilon_1}$ will be taken to be induced by the augmentations $\varepsilon_i$, and we denote its components by $d_{+-}$, $d_{+0}$, $\ldots$, etc.

Our first goal is establishing the following long exact sequence (or, in the ungraded case, exact triangle) which exists in the above setting. Its existence depends heavily on the fact that $\Sigma_{\{\Lambda^s\}}$ is the trace of a \emph{positive} Legendrian loop.
\begin{Thm}\label{thm:les}
\begin{enumerate}
\item In the case when $(M,\alpha)$ is a contactisation $(P \times \R,\alpha_{\OP{std}})$ endowed with its standard symplectic form, then there exists a long exact sequence
\begin{equation}
\label{eq:les}
\xymatrixrowsep{0.15in}
\xymatrixcolsep{0.15in}
\xymatrix{
       \cdots\ar[r]& LCH^{k-1+\mu}_{\varepsilon_0^+,\varepsilon_1^+}(\Lambda_0,\Lambda_1) \ar[d]^{\delta} & & & \\
& HF_{\varepsilon_0,\varepsilon_1}^{k}(\Sigma_{\{\Lambda^s\}},\R \times \Lambda_1) \ar[r]^{d_{-0}} & LCH^{k}_{\varepsilon_0,\varepsilon_1}(\Lambda_0,\Lambda_1) \ar[r]^{d_{+-}} & LCH^{k+\mu}_{\varepsilon_0^+,\varepsilon_1^+}(\Lambda_0,\Lambda_1)\ar[r] &\cdots,}
\end{equation}
for some fixed $\mu \in \Z,$ in which
$$ \delta := [\pi] \circ \left[\begin{pmatrix} d_{+0} & d_{+-}
\end{pmatrix}\right]^{-1},$$
with $\pi \colon CF^*(\Sigma_1,\Sigma_0) \oplus C^*_{-\infty} \to CF^*(\Sigma_0,\Sigma_1)$ being the canonical projection. 
\item In the case when the Legendrian submanifolds $\Lambda_0,\Lambda_1 \subset (M,\alpha)$ are hypertight, there again exists an analogous long exact sequence satisfying the same properties, but where the middle and rightmost terms are replaced by $LCH^{\alpha,k}(\Lambda_0,\Lambda_1)$ and $LCH^{\beta,k+\mu}(\Lambda_0,\Lambda_1),$ respectively, for suitable homotopy classes $\alpha,\beta \in \pi_0(\Pi(\R \times M;\R \times \Lambda_0,\R \times \Lambda_1))$. In this case we, moreover, have $d_{-0}=0$.
\end{enumerate}
\end{Thm}
\begin{Rem}
\label{rem:les}
Here we give some explanations concerning the formulation of the above theorem.
\begin{itemize}
\item In either of the Cases (1) and (2) above, there may be a relative difference of Maslov potentials of the Legendrians at the positive and negative ends induced by the cobordism $\Sigma_{\{\Lambda^s\}}$. Such a difference would induce a nonzero shift of grading by $\mu \in \Z$ in the rightmost homology group.
\item In Case (2), i.e.~the hypertight case, we only consider the canonical augmentation for each of $\Lambda_i$, $i=0,1,$ which obviously is preserved under the pull-back by a DGA morphism induced by any Lagrangian concordance. Unless the isotopy is homotopically trivial, the homotopy classes can indeed be different, i.e.~$\alpha \neq \beta.$ In this case, fixing one of them also determines the other one. Also, see Remark \ref{rem:htpyclass}.
\end{itemize}
\end{Rem}

\begin{proof}
The Cthulhu differential takes the form 
$$\mathfrak{d}_{\varepsilon_0,\varepsilon_1}=\begin{pmatrix}
  d_{++} & d_{+0} & d_{+-}\\
0 & d_{00} & 0 \\
0 & d_{-0} & d_{--}
\end{pmatrix},$$
since $d_{0-}=0$ by the action computation in Proposition \ref{prp:action} which we have postponed to Section \ref{sec:chtulhu-compl-assoc} below. In other words, the complex $\OP{Cth}_*(\Sigma_{\{\Lambda^s\}},\R \times \Lambda_1)$ is the mapping cone of the chain map
$$ \begin{pmatrix}d_{+0} & d_{+-} \end{pmatrix} \colon CF(\Sigma_0,\Sigma_1) \oplus C_{-\infty}(\Sigma_0,\Sigma_1) \to C_{+\infty}(\Sigma_0,\Sigma_1),$$
whose domain, in turn, is the mapping cone of
$$ d_{-0} \colon CF(\Sigma_0,\Sigma_1) \to C_{-\infty}(\Sigma_0,\Sigma_1).$$
By Theorem \ref{thm:invariance0} (in Case (1)) or Theorem \ref{thm:invariance2} (in Case (2)) the total complex is moreover acyclic. For Case (2) we must here use the fact that the trace of a Legendrian isotopy as constructed in Section \ref{sec:posit-isot-lagr} always is an invertible Lagrangian cobordism, as follows from Proposition \ref{prp:contractible}.

The existence of the long exact sequence is now standard consequence of this double cone structure, together with the acyclicity of the total complex.
\end{proof}
In certain particular situations the above long exact sequence degenerates into the statement that $d_{+-}$ is an isomorphism.
\begin{Thm}
\label{thm:delta}
If assumptions (1) and (2) of Theorem \ref{thm:les} are strengthened to
\begin{enumerate}
\item the Legendrian submanifolds $\Lambda_0,\Lambda_1 \subset (P \times \R,\alpha_{\OP{std}})$ satisfy the properties that no Reeb chord starts on $\Lambda_0$ and ends on $\Lambda_1$,
\item in the hypertight case, the positive loop of Legendrians is contractible amongst Legendrian loops,
\end{enumerate}
respectively, then it moreover follows that $d_{+-}$ is an isomorphism or, equivalently, that $d_{-0}=\delta=0$.
\end{Thm}
\begin{proof}First we note that in Case (1) we can `wrap' the inner part of the cobordism $\Sigma_1$ by an application of the negative Reeb flow, thereby removing all intersection points with $\Sigma_0$. More precisely, we can use e.g.~the Hamiltonian flow $\phi^s_{-\beta\partial_z}$ generated by the Hamiltonian vector field $-\beta(t)\partial_z$ for the compactly supported bump-function $\beta \colon \R \to \R_{\ge 0}$ shown in Figure \ref{fig:bulge}. See Section \ref{sec:wrapping} for more details. We thus take $\Sigma_0':=\Sigma_0$ and $\Sigma_1':=\phi^s_{-\beta\partial_z}(\Sigma_1)$, and then note that $\Sigma_0' \cap \Sigma_1'=\emptyset$ whenever $s \gg 0$ is taken to be sufficiently large.

In Case (2) a compactly supported Hamiltonian of which removes all intersection points between the cobordisms exists by the contractibility of the positive loop; combined with Proposition \ref{prp:contractible} we deduce that $\Sigma_0$ is compactly supported Hamiltonian isotopic to the trivial cylinder $\Sigma_0':=\R \times \Lambda_0$. In this case, we write $\Sigma_1':=\Sigma_1$, and observe that again $\Sigma_0' \cap \Sigma_1' = \emptyset$.

In either of the Cases (1) and (2), the result is now a consequence of the refined invariance results Corollary \ref{cor:invariancerefined} (in the case of a contactisation) and Theorem \ref{thm:invarianceht} (in the hypertight case); the previously established Hamiltonian isotopies imply that the assumptions of these theorems indeed are satisfied. We proceed with the details.

First, the refined invariance results give us $d_{-0}=0$. To show $\delta=0$, we consider the quasi-isomorphism
\[\begin{pmatrix} d_{+0} & d_{+-} \end{pmatrix} \colon CF^*(\Sigma_0,\Sigma_1) \oplus C_{-\infty}^*(\Sigma_0,\Sigma_1) \to  C_{+\infty}^*(\Sigma_0,\Sigma_1),\]
where the latter is a subcomplex of $\Cth_*(\Sigma_0,\Sigma_1)$, and where the former is the corresponding quotient complex. (The property of being a quasi-isomorphism is equivalent to the acyclicity of $\Cth_*(\Sigma_0,\Sigma_1)$.) The vanishing $\delta=0$ will be established by showing that the restriction
$$\begin{pmatrix} d_{+0} & d_{+-} \end{pmatrix}|_{CF^*(\Sigma_0,\Sigma_1)} \colon (CF^*(\Sigma_0,\Sigma_1),d_{00}) \to (C_{+\infty}^*(\Sigma_0,\Sigma_1),d_{++})$$
vanishes in homology; c.f.~the definition of $\delta$ in the formulation of Theorem \ref{thm:les}. 

For any cycle $a \in (CF^*(\Sigma_0,\Sigma_1),d_{00})$, the vanishing $d_{-0}=0$ implies that
\[ \begin{pmatrix} d_{+0} & d_{+-} \end{pmatrix}
\begin{pmatrix}
  a \\ 0
\end{pmatrix}
=\mathfrak{d}_{\varepsilon_0,\varepsilon_1}(a) \in C_{+\infty}^*(\Sigma_0,\Sigma_1),\]
i.e.~the quasi-isomorphism restricted to the intersection points coincides with the differential of the Cthulhu complex.

The aforementioned invariance results Corollary \ref{cor:invariancerefined} and Theorem \ref{thm:invarianceht}, moreover, produce a homotopy equivalence $\phi$ from $\Cth_*(\Sigma_0,\Sigma_1)$ to the complex $\Cth_*(\Sigma_0',\Sigma_1')$, where the latter has no generators corresponding to intersection points. We then use the chain map property together with the fact that $\phi(a) \in C_{+\infty}^*(\Sigma_0',\Sigma_1')$ (c.f.~the refined invariance results) in order to deduce that
$$ \phi(\mathfrak{d}_{\varepsilon_0,\varepsilon_1}(a))=d_{\varepsilon_0',\varepsilon_1'} \circ \phi(a).$$
In other words, the image
$$\phi \left(\begin{pmatrix} d_{+0} & d_{+-} \end{pmatrix}
\begin{pmatrix}
  a \\ 0
\end{pmatrix}\right) \in C_{+\infty}^*(\Sigma_0',\Sigma_1')$$
vanishes in homology. Since, moreover, $\phi_+ \colon C_{+\infty}^*(\Sigma_0,\Sigma_1) \to C_{+\infty}^*(\Sigma_0',\Sigma_1')$ is a chain isomorphism (again, c.f.~the refined invariance), we thus conclude that
$$\begin{pmatrix} d_{+0} & d_{+-} \end{pmatrix}\begin{pmatrix}
  a \\ 0
\end{pmatrix} \in C_{+\infty}^*(\Sigma_0,\Sigma_1)$$
itself vanishes in homology as sought.
\end{proof}

\subsection{Action computations}
\label{sec:chtulhu-compl-assoc}

The assumption that the isotopy is positive implies that
\begin{Prop}
\label{prp:action}
The generators of $CF^*(\Sigma_{\{\Lambda^s\}},\R \times \Lambda_1)$ are all of \emph{negative} action. In particular, the term $d_{0-}$ in the Cthulhu differential vanishes.
\end{Prop}
\begin{proof}
The first statement follows from the computation in Proposition \ref{prp:pos_cyl}. Observe that the potential $f_{\R \times \Lambda_1}\equiv 0$ necessarily vanishes by our conventions. Since the Reeb chord generators at the negative end are of positive action, the fact that $d_{0-}=0$ is now an immediate consequence of Lemma \ref{lem:action}.
\end{proof}

The following simple action computation will also be used repeatedly.
\begin{Lem}
\label{lem:actionreeb}
When computing the action in $\OP{Cth}_*(\Sigma_{\{\Lambda^s\}},\R \times \Lambda_1)$, we may take $T_-=0$ and $T_+=T$ (i.e.~the constant from the construction in Section \ref{sec:posit-isot-lagr}). With these conventions, a Reeb chord generator $\gamma^+ \in C_{+\infty}^*(\Sigma_{\{\Lambda^s\}},\R \times \Lambda_1)$ has action equal to
\[\mathfrak{a}(\gamma^+)=e^{T}\ell(\gamma^+)-f^+_{\Sigma_{\{\Lambda^s\}}},\]
with a positive constant $f^+_{\Sigma_{\{\Lambda^s\}}}>0$.
\end{Lem}
\begin{proof}
The positivity $f^+_{\Sigma_{\{\Lambda^s\}}}>0$ follows from Proposition \ref{prp:pos_cyl}, while $f^+_{\R \times \Lambda_1} \equiv 0$ is a consequence of our conventions.
\end{proof}

\subsection{The proof of Theorem \ref{thm:hypertight} (the hypertight case)}
\label{sec:proofhypertight}
We argue by contradiction. Let $\Lambda_0:=\Lambda$ be our given hypertight Legendrian submanifold and let $\Sigma_{\{\Lambda^s\}}$ be the cylinder induced by a \emph{contractible} positive loop containing $\Lambda$. We take $\Lambda_1$ to be obtained from $\Lambda$ by, first, applying the time-$(-\epsilon)$ Reeb flow and, second, perturbing the resulting Legendrian by the one-jet $j^1f \subset \mathcal{J}^1\Lambda$ inside a standard Legendrian neighbourhood (in which $\Lambda$ is identified with the zero-section). Here $f \colon \Lambda \to [-\epsilon,0]$ is assumed to be a Morse function.
\begin{Lem}
\label{lem:d+-}
For $\epsilon>0$ sufficiently small, the generators of the subcomplex and quotient complex
$$C_{\pm\infty}^*(\Sigma_{\{\Lambda^s\}},\R \times \Lambda_1) \subset \Cth_*^0(\Sigma_{\{\Lambda^s\}},\R \times \Lambda_1)$$
correspond bijectively to the critical points of the above Morse function $f$. Further, the generators of the subcomplex $C_{+\infty}^*(\Sigma_{\{\Lambda^s\}},\R \times \Lambda_1)$ are of negative action, while the generators of $C_{-\infty}^*(\Sigma_{\{\Lambda^s\}},\R \times \Lambda_1)$ are of positive action.
\end{Lem}
\begin{proof}
There is a correspondence between critical points $\{p\} \subset \Lambda$ of $f$ and a subset of the Reeb chords from $\Lambda_1$ to $\Lambda_0$. Moreover, the length of the Reeb chord corresponding to the critical point $p \in \Lambda$ is equal to $-f(p)+\epsilon \le 2\epsilon$.

First, using the assumption that $\{\Lambda^s\}$ is contractible, and hence that $\Sigma_{\{\Lambda^s\}}$ is Hamiltonian isotopic to a trivial cylinder by Proposition \ref{prp:contractible}, it follows that all these Reeb chords also are generators of $C_{+\infty}^*$. Second, using the assumption of hypertightness (i.e.~that there are no contractible Reeb chords on $\Lambda$), it follows that these are \emph{all} of the generators of $C_{+\infty}^*$.  Here we recall that $\Cth_*^0(\Sigma_{\{\Lambda^s\}},\R \times \Lambda_1)$ is generated by only those generators which live in the `contractible' homotopy class; also see Remark \ref{rem:htpyclass}. 

The negativity of the action is then a consequence of the inequality
$$ \mathfrak{a}(\gamma^+) \le e^T(2\epsilon)-f_{\Sigma_{\{\Lambda^s\}}}^+,$$
given that $\epsilon>0$ is chosen sufficiently small,  and where the primitive $f_{\Sigma_{\{\Lambda^s\}}}^+$ at $t = +\infty$ is positive by the positivity of the loop of Legendrians; see Lemma \ref{lem:actionreeb}.
\end{proof}

\begin{Lem}
\label{lem:morse}
For $\epsilon>0$ sufficiently small, the subcomplex and quotient complex
$$C_{\pm\infty}^*(\Sigma_{\{\Lambda^s\}},\R \times \Lambda_1) \subset \Cth_*^0(\Sigma_{\{\Lambda^s\}},\R \times \Lambda_1)$$
both compute the Morse homology of $\Lambda$.
\end{Lem}
\begin{proof}
The identification on the level of generators follows from Lemma \ref{lem:d+-}. The differential can be identified with the Morse differential, following a standard computation that carries over from the computation made `locally' in the jet-space of $\Lambda$ using the theory from e.g.~\cite{MorseFlow}. To that end, we must choose the almost complex structure in some neighbourhood $\R \times U$ appropriately, for $U \subset M$ of $\Lambda$ contactomorphic to a neighbourhood of the zero-section of $\mathcal{J}^1\Lambda$.  In particular, in that neighbourhood we want the almost complex structure to be a cylindrical lift of the almost complex structure on $T^*\Lambda$ that is produced by \cite{MorseFlow}.  (The monotonicity property for the symplectic area of pseudoholomorphic discs can then be used to ensure that the strips in the differential do not leave the neighbourhood $\R \times U \subset \R \times M$ of $\R \times \Lambda$ in the symplectisation; see e.g.~the proof of \cite[Lemma 6.4]{LiftingPseudoholomorphic}.)
\end{proof}
We now prove Theorem \ref{thm:hypertight} by invoking the above lemmas in conjunction with Theorem \ref{thm:les}.
\begin{proof}[Proof of Theorem \ref{thm:hypertight}]
  The rightmost term $H(C_{+\infty}^*,d_{++})$ in the long exact sequence produced by Theorem \ref{thm:les} is non-zero by Lemma \ref{lem:morse}, while the map $d_{+-}$ in this long exact sequence vanishes by Lemma \ref{lem:d+-}  combined with Lemma \ref{lem:action}.  This is in contradiction with the fact that $d_{+-}$ is an isomorphism, as established by Theorem \ref{thm:delta}. 
\end{proof}

\subsection{The proof of Theorem \ref{thm:hypertight2}}
We consider the setup of the proof of Theorem \ref{thm:hypertight} given in Section \ref{sec:proofhypertight} above, but where $\Sigma_{\{\Lambda^s\}}$ is not necessarily compactly supported Hamiltonian isotopic to a trivial cylinder. The only difference with the case above is that the consequences of Lemmas \ref{lem:d+-} and \ref{lem:morse} might not hold for the subcomplex $C_{+\infty}^*(\Sigma_{\{\Lambda^s\}},\R \times \Lambda_1)$. However, we note that it always is the case that $C_{-\infty}^*(\Sigma_{\{\Lambda^s\}},\R \times \Lambda_1)$ is the Morse homology complex of $\Lambda$.

The reason for why Lemmas \ref{lem:d+-} and \ref{lem:morse} can fail is that, depending on the homotopy properties of $\Sigma_{\{\Lambda^s\}}$, it is possible that the subcomplex
$$C_{+\infty}^*(\Sigma_{\{\Lambda^s\}},\R \times \Lambda_1) \subset \Cth_*(\Sigma_{\{\Lambda^s\}},\R \times \Lambda_1)$$
is in fact generated by chords corresponding to the non-contractible Reeb chords on $\Lambda$ in some fixed homotopy class $\alpha \neq 0$ (see part (2) of Remark \ref{rem:les}). 

Notwithstanding, when $\alpha \neq 0$ we can use our assumptions on the Conley--Zehnder indices in order to show that $d_{+-}$ is not injective in this case either (thus leading to a contradiction). Namely, the map $d_{+-}$ restricted to the two-dimensional subspace of
$$H(C_{-\infty}^*(\Sigma_{\{\Lambda^s\}},\R \times \Lambda_1),d_{--})$$
that is generated by the maximum and the minimum of the Morse function must have a non-trivial kernel. Indeed, the maximum and the minimum are two non-zero classes of degrees that differ by precisely $\dim \Lambda=n$; however, by the assumptions of the theorem, the target homology group does not contain two classes with such an index difference unless $\alpha = 0$. In either case, the deduced non-injectivity is again in contradiction with $d_{-0}=0$ and the exactness of the sequence in Theorem \ref{thm:les}. \qed

\subsection{Spectral invariants for pairs of Legendrians}
Spectral invariants where introduced by Viterbo \cite{ViterboSpectral} and later developed by Oh \cite{OhSpectral}. They are now a well-established technique for studying quantitative questions in symplectic topology. They have also been defined for Legendrian submanifolds in certain contact manifolds by Zapolsky in \cite{Zapolsky}, and Sabloff--Traynor considered some of their properties under Lagrangian cobordisms in their work \cite{Sabloff_Traynor_2}. Here we study further properties that are satisfied under  Lagrangian cobordisms and positive isotopies which will be used when proving Theorem \ref{thm:liouville}. Since that theorem concerns the contactisation of a Liouville domain, we will for simplicity restrict ourselves to that geometric setting in this subsection.

For any pair of Legendrian submanifolds $\Lambda_i \subset (M,\alpha)$, $i=0,1$, together with a pair $\varepsilon_i$ of augmentations, we consider the canonical inclusion
$$\iota_\ell \colon LCC_{\varepsilon_0,\varepsilon_1}^*(\Lambda_0,\Lambda_1)^{[\ell,+\infty]} \: \subset \: LCC_{\varepsilon_0,\varepsilon_1}^*(\Lambda_0,\Lambda_1) $$
of the subcomplex spanned by the Reeb chords being of length at least $\ell \ge 0$ and use $[\iota_\ell]$ to denote the induced map on the homology level.
\begin{defn}
The \emph{spectral invariant} of the pair $(\varepsilon_0,\varepsilon_1)$ of augmentations is defined to be
\begin{gather*}
 c_{\varepsilon_0,\varepsilon_1}(\Lambda_0,\Lambda_1):=\sup \left\{ \ell \in \R; \: \OP{coker} [\iota_{\ell}] = 0 \right\} \in (0,+\infty],
\end{gather*}
which is a finite positive real number if and only if $LCH_{\varepsilon_0,\varepsilon_1}^*(\Lambda_0,\Lambda_1) \neq 0$.
\end{defn}
Note that the latter property holds since, for a closed Legendrian submanifold of a contactisation, the Legendrian contact homology complex is generated by finitely many Reeb chords (which have positive length) and thus for $\ell \gg 0$ it is the case that $LCC_{\varepsilon_0,\varepsilon_1}^*(\Lambda_0,\Lambda_1)^{[\ell,+\infty]}=0$.

By construction we obtain a non-zero homology class whenever the spectral capacity is finite, namely:
\begin{Lem}

\label{lem:spectral}
\begin{enumerate}
\item All non-zero homology classes in $LCH_{\varepsilon_0,\varepsilon_1}^*(\Lambda_0,\Lambda_1)$ can be represented by a linear combination of generators being of length at least $c_{\varepsilon_0,\varepsilon_1}(\Lambda_0,\Lambda_1).$
\item If $LCH_{\varepsilon_0,\varepsilon_1}^*(\Lambda_0,\Lambda_1) \neq 0,$ then there exists a non-zero class
$$\alpha_{\varepsilon_0,\varepsilon_1}(\Lambda_0,\Lambda_1) \in \im [\iota_{c_{\varepsilon_0,\varepsilon_1}(\Lambda_0,\Lambda_1)}] \in LCH_{\varepsilon_0,\varepsilon_1}^*(\Lambda_0,\Lambda_1)$$
which is not in the image of $[\iota_\ell]$ for any $\ell > c_{\varepsilon_0,\varepsilon_1}(\Lambda_0,\Lambda_1)$.
\end{enumerate}
\end{Lem}

The following propositions give the crucial behaviour for our spectral invariant under the relation of Lagrangian cobordisms.

Let $\Sigma_0=\Sigma_{\{\Lambda^s\}}$ be a concordance from $\Lambda^-_0$ to $\Lambda^+_0$ induced by a Legendrian isotopy as constructed in Section \ref{sec:posit-isot-lagr}, while $\Sigma_1:=\R \times \Lambda_1$ is a trivial Lagrangian cylinder over a Legendrian $\Lambda_1 = \Lambda_1^\pm$.
\begin{Prop}
\label{prp:spectral}
For any choice of augmentations $\varepsilon_i$ of $\Lambda_i^+$, $i=0,1$, there exists augmentations $\varepsilon_i^-$ of $\Lambda_i^-$ for which
$$ c_{\varepsilon_0^+,\varepsilon_1^+}(\Lambda_0^+,\Lambda_1^+)=c_{\varepsilon_0,\varepsilon_1}(\Lambda_0^+,\Lambda_1^+)$$
is satisfied with $\varepsilon^+_i:=\varepsilon^-_i \circ \Phi_{\Sigma_i}$. 
\end{Prop}

\begin{proof}[Proof of Proposition \ref{prp:spectral}]
We prove this by using Part (1) of the invariance result Theorem \ref{thm:invariance1} combined with a neck-stretching argument. The augmentation $\varepsilon_0^-$ will be taken to be equal to $\varepsilon_0 \circ \Phi_{\Sigma_{\{\Lambda^{1-s}\}}}$.

To show the statement we consider the concatenation
$$ \widetilde{\Sigma} := \Sigma_{\{\Lambda^{1-s}\}} \odot \Sigma_{\{\Lambda^s\}} \subset (\R \times M,d(e^t\alpha))$$
of Lagrangian concordances, which is compactly supported Hamiltonian isotopic to the trivial cylinder $\R \times \Lambda_0$ by Proposition \ref{prp:contractible}. After a neck-stretching along the hypersurface $\{ t_0 \} \times M$ along which the two cobordisms in the concatenation are joined, together with a gluing argument (see Section \ref{sec:stretching}), we can establish the last equality in
\begin{gather*}
 \varepsilon_0^+:=\varepsilon_0^- \circ \Phi_{\Sigma_{\{\Lambda^s\}}} = \varepsilon_0 \circ \Phi_{\Sigma_{\{\Lambda^{1-s}\}}} \circ \Phi_{\Sigma_{\{\Lambda^s\}}}= \varepsilon_0 \circ \Phi_{\widetilde{\Sigma}}.
\end{gather*}
See e.g.~\cite[Lemma 5.4]{cthulhu} for a similar result.

The claim now follows from the existence of the isomorphism
$$\phi_+ \colon C_{+\infty}^*(\R \times \Lambda_0,\R \times \Lambda_1) \to C_{+\infty}^*(\widetilde{\Sigma},\R \times \Lambda_1),$$
established in Part (1) of Theorem \ref{thm:invariance1}, where
\begin{eqnarray*}
& & C_{+\infty}^*(\R \times \Lambda_0,\R \times \Lambda_1)=LCH_{\varepsilon_0,\varepsilon_1}^*(\Lambda_0,\Lambda_1),\\
& & C_{+\infty}^*(\widetilde{\Sigma},\R \times \Lambda_1)=LCH_{\varepsilon_0^+,\varepsilon_1^+}^*(\Lambda_0,\Lambda_1).
\end{eqnarray*}
More precisely, the latter theorem has here been applied to the complex $(\Cth_*(\R \times \Lambda_0,\R \times \Lambda_1),\mathfrak{d}_{\varepsilon_0,\varepsilon_1})$ while using the existence of the previously established Hamiltonian isotopy from $\R \times \Lambda_0$ to $\widetilde{\Sigma}$. Here it is crucial that $\phi_+$, and hence $\phi_+^{-1}$ as well, are upper triangular with respect to the action filtration, from which one readily deduces that
$$ c_{\varepsilon_0^+,\varepsilon_1^+}(\Lambda_0^+,\Lambda_1^+)=c_{\varepsilon_0,\varepsilon_1}(\Lambda_0^+,\Lambda_1^+)$$
holds as sought.
\end{proof}

\begin{Thm}
\label{thm:spectral}
Further assume that $\Sigma_0 \cap \Sigma_1 = \emptyset$ and that the Legendrian isotopy $\{\Lambda^s\}$ is positive. For any pair of augmentations $\varepsilon_i^-$ of $\Lambda_i^-$ it is then the case that
$$ c_{\varepsilon_0^+,\varepsilon_1^+}(\Lambda_0^+,\Lambda_1^+)>c_{\varepsilon_0^-,\varepsilon_1^-}(\Lambda_0^-,\Lambda_1^-)$$
for the pull-back augmentations $\varepsilon^+_i:=\varepsilon^-_i \circ \Phi_{\Sigma_i}$. Moreover, $c_{\varepsilon_0^+,\varepsilon_1^+}(\Lambda_0^+,\Lambda_1^+)=+\infty$ holds if and only if $c_{\varepsilon_0^-,\varepsilon_1^-}(\Lambda_0^-,\Lambda_1^-)=+\infty$.
\end{Thm}
The following subsection will be devoted to the proof of this theorem.

\subsection{Proof of Theorem \ref{thm:spectral}}
There are two possible strategies for proving this theorem. One is by studying the invariance properties of the Legendrian contact homology complex under a Legendrian isotopy, and one involves studying the Floer homology of the Lagrangian cobordism corresponding to the trace of the isotopy as constructed in Section \ref{sec:posit-isot-lagr}. Here we take the latter approach.

For technical reasons we will in the following need to make use of the so-called \emph{cylindrical lift} $J_P$ of a tame almost complex structure $J_P$ on $(P,d\theta)$; these are cylindrical almost complex structures for which the canonical projection $\Pi \colon \R \times P \times \R \to P$ is $(\widetilde{J}_P,J_P)$-holomorphic. In particular, such an almost complex structure $\widetilde{J}_P$ is invariant under translations of both the $t$ and $z$ coordinates.

First we perform some general computations for a pair $\Lambda_i$, $i=0,1$, of Legendrians together with augmentations $\varepsilon_i$. By $\Lambda^s_1$ we denote the image of $\Lambda_1$ under the time-$s$ Reeb flow, i.e.~translation of the $z$ coordinate by $s \in \R$. 
\begin{Lem}
\label{lem:alpha0}
For a suitable cylindrical lift of an almost complex structure on $P$ and generic $\ell \ge 0$, the canonical isomorphism
$$LCC_{\varepsilon_0,\varepsilon_1}^*(\Lambda_0,\Lambda_1)^{[\ell,+\infty]} = LCC_{\varepsilon_0,\varepsilon_1}^*(\Lambda_0,\Lambda_1^\ell)$$
is an isomorphism of complexes.
\end{Lem}
\begin{proof}
The fact that the differentials agree follows from \cite[Theorem 2.1]{LiftingPseudoholomorphic}. Namely, the differentials of both complexes can be computed in terms of $J_P$-holomorphic discs inside $P$ that have boundary on the Lagrangian projection
$$\Pi((\R \times \Lambda_0) \cup (\R \times \Lambda_1))=\Pi((\R \times \Lambda_0) \cup (\R\times \Lambda^\ell_1)) \subset (P,d\theta),$$
which is an exact Lagrangian immersion. In particular, the translation in the $z$-coordinate does not play a role for the counts of the relevant pseudoholomorphic discs.
\end{proof}
Now consider the exact Lagrangian cobordism $\phi^\ell_{\beta^+\partial_z}(\R\times\Lambda_1^+)$ from $(\Lambda_1^+)^\ell$ to $\Lambda_1^+$ induced by wrapping the trivial cylinder, where $\beta^+(t)$ is the function described in Section \ref{sec:wrapping}.
\begin{Lem}
\label{lem:iota}
Consider the complex $\Cth_*(\R \times \Lambda_0^+,\phi^\ell_{\beta^+\partial_z}(\R\times\Lambda_1^+))$ with differential $\mathfrak{d}_{\varepsilon_0^+,\varepsilon_1^+}^{(+)}$. The cylindrical lift $\widetilde{J}_P$ may be assumed to be regular for the strips in the differential of this complex, for which we compute
\begin{eqnarray*}
& & C^*_{-\infty}(\R \times \Lambda_0^+,\phi^\ell_{\beta^+\partial_z}(\R\times\Lambda_1^+))=LCC_{\varepsilon_0^+,\varepsilon_1^+}^*(\Lambda_0^+,(\Lambda_1^+)^\ell),\\
& & C^*_{+\infty}(\R \times \Lambda_0^+,\phi^\ell_{\beta^+\partial_z}(\R\times\Lambda_1^+))=LCC_{\varepsilon_0^+,\varepsilon_1^+}^*(\Lambda_0^+,\Lambda_1^+),
\end{eqnarray*}
and there is an equality $\varepsilon^+_1 \circ \Phi_{\phi^\ell_{\beta_+\partial_z}(\R \times \Lambda_1^+)} = \varepsilon^+_1$ of augmentations. Furthermore, the component $d_{+-}^{(+)}$ of $\mathfrak{d}_{\varepsilon_0^+,\varepsilon_1^+}^{(+)}$ is equal to the canonical inclusion $\iota_\ell$, i.e.
$$d^{(+)}_{+-}=\iota_\ell \colon LCC_{\varepsilon_0^+,\varepsilon_1^+}^*(\Lambda_0^+,(\Lambda_1^+)^{\ell}) \to LCC_{\varepsilon_0^+,\varepsilon_1^+}^*(\Lambda_0^+,\Lambda_1^+).$$
(See Lemma \ref{lem:alpha0} for the identifications used here.) 
\end{Lem}
\begin{proof}
The computations are analogous to those carried out in \cite[Lemma 2.10]{Floer_Conc}. We here give the details only for how to establish the identification
$$ \varepsilon^+_1 \circ \Phi_{\phi^\ell_{\beta_+\partial_z}(\R \times \Lambda_1^+)} = \varepsilon^+_1,$$
where the left-hand side is the pull-back of the augmentation $\varepsilon^+_1$ under the DGA morphism induced by the cobordism $\phi^\ell_{\beta_+\partial_z}(\R \times \Lambda_1^+)$. The remaining claims follow by similar arguments.

It suffices to show that the DGA morphism induced by the cobordism $\phi^\ell_{\beta^+\partial_z}(\R \times \Lambda_1^+)$ is the canonical identification of complexes. This is indeed the case, given the assumption that the DGA morphism is defined using an almost complex structure that is a cylindrical lift. Namely, the image under the canonical projection
$$ \Pi \colon \R \times P \times \R \to P$$
of the discs in the definition of the DGA morphism are $J_P$-holomorphic discs having boundary on $\Pi(\R \times \Lambda_1^+)$. By a simple index computation (see e.g.~\cite[Lemma 8.3]{LiftingPseudoholomorphic}) these discs are, moreover, of \emph{negative} expected dimension and must hence be constant (under the assumption that $J_P$ is regular).

Finally, the discs in the definition of the DGA morphism can even be seen to bijectively correspond to the double points of the Lagrangian projection $\Pi((\R \times \Lambda_0^+) \cup (\R \times \Lambda_1^+)) \subset (P,d\theta)$; namely, there is an explicit and uniquely defined rigid $\widetilde{J}_P$-holomorphic disc in $\R \times P \times \R$ contributing to the DGA morphism living above each such double point. The regularity of the latter explicitly defined discs was shown in \cite[Lemma 8.3]{LiftingPseudoholomorphic}.
\end{proof}
Similar computations can be made concerning the exact Lagrangian cobordism $\phi^\ell_{\beta^-\partial_z}(\R\times\Lambda_1^-)$ from $\Lambda_1^-$ to $(\Lambda_1^-)^\ell$ (again, see Section \ref{sec:wrapping}). Namely, we have:
\begin{Lem}
\label{lem:augs}
Consider the complex $\Cth_*(\R \times \Lambda_0^-,\phi^\ell_{\beta^-\partial_z}(\R\times\Lambda_1^-))$ with differential $\mathfrak{d}^{(-)}_{\varepsilon_0^-,\varepsilon_1^-}$. The cylindrical lift $\widetilde{J}_P$ may be assumed to be regular for the strips in the differential of this complex, for which we compute
\begin{eqnarray*}
& & C^*_{-\infty}(\R \times \Lambda_0^-,\phi^\ell_{\beta^-\partial_z}(\R\times\Lambda_1^-))=LCC_{\varepsilon_0^-,\varepsilon_1^-}^*(\Lambda_0^-,\Lambda_1^-),\\
& & C^*_{+\infty}(\R \times \Lambda_0^-,\phi^\ell_{\beta^-\partial_z}(\R\times\Lambda_1^-))=LCC_{\varepsilon_0^-,\varepsilon_1^-}^*(\Lambda_0^-,(\Lambda_1^-)^\ell),
\end{eqnarray*}
and there is an equality $\varepsilon_1^- \circ \Phi_{\phi^\ell_{\beta_-\partial_z}(\R \times \Lambda_1)} = \varepsilon_1^-$ of augmentations. (See Lemma \ref{lem:alpha0} for the identifications made here.)
\end{Lem}
Now we consider the Lagrangian cobordisms $\Sigma_0=\Sigma_{\{\Lambda^s\}}$ and $\Sigma_1=\phi^\ell_{\beta\partial_z}(\R \times \Lambda_1)$, where the function $\beta(t)$ is as described in Section \ref{sec:wrapping}; i.e.~$\Sigma_i$ is a cobordism from $\Lambda_i^-=\Lambda_i$ to $\Lambda_i^+=\Lambda_i$. Consider the non-zero element $\alpha_{\varepsilon_0^+,\varepsilon_1^+}(\Lambda_0^+,\Lambda_1^+) \in LCH_{\varepsilon_0^+,\varepsilon_1^+}^*(\Lambda_0^+,\Lambda_1^+)$ as in Part (2) of Lemma \ref{lem:spectral} for $\varepsilon_i^+=\varepsilon_i^- \circ \Phi_{\Sigma_i}$. Here  Recall that each of its representatives is a linear combination of Reeb chord generators of which at least one is of length equal to $c_{\varepsilon_0^+,\varepsilon_1^+}(\Lambda_0^+,\Lambda_1^+)>0$.
\begin{Lem}
\label{lem:prestretch}
Assume that $\{\Lambda^s\}$ is a positive isotopy from $\Lambda_0^-$ to $\Lambda_0^+$, that $\Sigma_{\{\Lambda^s\}} \cap (\R \times \Lambda_1) = \emptyset$, and that
$$0<\ell<c_{\varepsilon_0^-,\varepsilon_1^-}(\Lambda_0^-,\Lambda_1^-).$$
It follows that $\alpha_{\varepsilon_0^+,\varepsilon_1^+}(\Lambda_0^+,\Lambda_1^+) \neq 0 \in LCH_{\varepsilon_0^+,\varepsilon_1^+}^*(\Lambda_0^+,\Lambda_1^+)$  has a representative  which is in the image of the component $d_{+-}'$ of the differential $\mathfrak{d}_{\varepsilon_0^-,\varepsilon_1^-}'$ of the complex $\Cth(\Sigma_{\{\Lambda^s\}},\phi^\ell_{\beta\partial_z}(\R \times \Lambda_1))$.
\end{Lem}
\begin{proof}

Denote by $\mathfrak{d}''_{\varepsilon_0^-,\varepsilon_1^-}$ the differential of $\Cth(\Sigma_{\{\Lambda^s\}},\R \times \Lambda_1),$ whose components will be denoted by $d_{+-}''(a),\ldots,$ etc.

Take $a \in C_{-\infty}^*(\Sigma_{\{\Lambda^s\}},\R \times \Lambda_1)$ satisfying $\mathfrak{d}_{\varepsilon_0^-,\varepsilon_1^-}''(a)=d_{+-}''(a)$ as well as 
\begin{equation}
\label{eq:def}
[d_{+-}''(a)]=[(\phi_+)^{-1}](\alpha_{\varepsilon_0^+,\varepsilon_1^+}(\Lambda_0^+,\Lambda_1^+)),
\end{equation}
where $\phi_+$ is the chain isomorphism induced by Theorem \ref{thm:invariance0} applied to the compactly supported Hamiltonian isotopy from $\R \times \Lambda_1$ to $\phi^\ell_{\beta\partial_z}(\R \times \Lambda_1)$. Here we need to use the fact that the complex $\Cth(\Sigma_{\{\Lambda^s\}},\R \times \Lambda_1)$ is acyclic, which also is a consequence of the invariance result Theorem \ref{thm:invariance0}. The non-zero element $[a]\in LCH_{\varepsilon_0^-,\varepsilon_1^-}^*(\Lambda_0^-,\Lambda_1^-)$ can be represented by Reeb chords all being of length at least $c_{\varepsilon_0^-,\varepsilon_1^-}(\Lambda_0^-,\Lambda_1^-)$ by Part (1) of Lemma \ref{lem:spectral}, and we replace $a$ by such a representative.

An intersection point generator of $CF^*(\Sigma_{\{\Lambda^s\}},\phi^r_{\beta\partial_z}(\R \times \Lambda_1))$ can be computed to be of action equal to at most $\ell>0$ whenever $0 < r \le \ell$. Part (2) of Theorem \ref{thm:invariance1} with $\mathfrak{a}_0=\ell$ thus shows that the chain homotopy
$$\phi \colon (\Cth(\Sigma_{\{\Lambda^s\}},\R \times \Lambda_1),\mathfrak{d}_{\varepsilon_0^-,\varepsilon_1^-}'') \to (\Cth(\Sigma_{\{\Lambda^s\}},\phi^\ell_{\beta\partial_z}(\R \times \Lambda_1)),\mathfrak{d}_{\varepsilon_0^-,\varepsilon_1^-}')$$
satisfies
$$\phi(a)= (A_- ,A_+) \in C_{-\infty}^*(\Sigma_{\{\Lambda^s\}},\phi^\ell_{\beta\partial_z}(\R \times \Lambda_1)) \oplus C_{+\infty}^*(\Sigma_{\{\Lambda^s\}},\phi^\ell_{\beta\partial_z}(\R \times \Lambda_1)).$$
In addition,
$$\phi_+\circ d_{+-}''(a)=\phi \circ \mathfrak{d}_{\varepsilon_0^-,\varepsilon_1^-}''(a)=\mathfrak{d}_{\varepsilon_0^-,\varepsilon_1^-}'\circ \phi(a)=(d_{+-}'+d_{--}')A_-+d_{++}'A_+,$$
where the last equality follows from the above action considerations of $\phi(a)$. From this together with Equality \eqref{eq:def} we then conclude that
$$[d_{+-}'(A_-)]=\alpha_{\varepsilon_0^+,\varepsilon_1^+}(\Lambda_0^+,\Lambda_1^+)$$
holds as sought.
\end{proof}

\begin{Lem}
\label{lem:stretch}
Under the assumptions of Lemma \ref{lem:prestretch}, we consider the element $\alpha_{\varepsilon_0^+,\varepsilon_1^+}(\Lambda_0^+,(\Lambda_1^+)^\ell) \neq 0 \in LCH_{\varepsilon_0^+,\varepsilon_1^+}^*(\Lambda_0^+,(\Lambda_1^+)^\ell)$ induced by the identification in Lemma \ref{lem:alpha0}. It follows that $\alpha_{\varepsilon_0^+,\varepsilon_1^+}(\Lambda_0^+,(\Lambda_1^+)^\ell)$ satisfies the properties that
\begin{enumerate}
\item any representative is a linear combination of the basis of Reeb chords of which at least one chord of length $c_{\varepsilon_0^+,\varepsilon_1^+}(\Lambda_0^+,\Lambda_1^+)-\ell>0$ appears with a non-zero coefficient, and
\item  some representative  is contained in the image of the component $d_{+-}$ of the differential $\mathfrak{d}_{\varepsilon_0^-,\varepsilon_1^-}$ of the complex $\Cth(\Sigma_{\{\Lambda^s\}},\R \times (\Lambda_1)^\ell)$.
\end{enumerate}
\end{Lem}
\begin{proof}
(1): This is a straight-forward consequence of  Part (2) of Lemma \ref{lem:spectral}.

(2): This follows from analysing the possible breakings when stretching the neck along the disconnected hypersurface
$$\{0,T\} \times P \times \R \subset (\R \times P \times \R,d(e^t\alpha_{\OP{std}}))$$
of contact type (see Section \ref{sec:wrapping}).

Namely, let us consider a generator $A_- \in C^*_{-\infty}(\Sigma_{\{\Lambda^s\}},\phi^\ell_{\beta\partial_z}(\R \times \Lambda_1))$ which gives a non-zero contribution $\langle d_{+-}'(A_-),\iota_\ell(\alpha_0) \rangle \neq 0$, where $\iota_\ell(\alpha_0)$ is a generator appearing with a non-zero coefficient in a representative $\alpha_{\varepsilon_0^+,\varepsilon_1^+}(\Lambda_0^+,(\Lambda_1^+))$; this is possible by Lemma \ref{lem:prestretch} combined with  Part (1) of Lemma \ref{lem:spectral}.  The limit of a pseudoholomorphic disc of this type when stretching the neck is shown in Figure \ref{fig:stretched}. The pseudoholomorphic curve in the middle level corresponds to a non-zero contribution to $\langle d_{+-}'(b),\alpha_0 \rangle \neq 0$, and the claim now follows.
\end{proof}

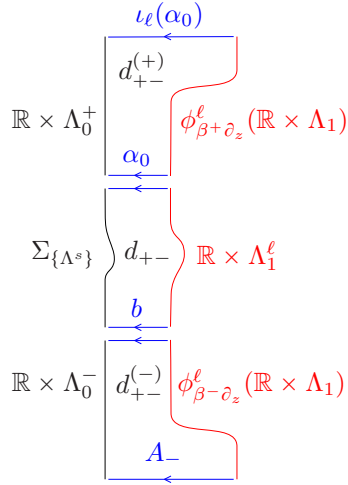
\begin{figure}[ht!]
\vspace{5mm}
\labellist
\pinlabel{$\mathbb{R}\times \Lambda_0^+$} [tr] at 0 465
\pinlabel{$\color{red}\phi^\ell_{\beta^+\partial_z}(\mathbb{R}\times\Lambda_1)$} [tl] at 85 465
\pinlabel{$d_{+-}^{(+)}$} [tl] at 10 530
\pinlabel{$\color{blue}\iota_\ell(\alpha_0)$} [tl] at 30 590
\pinlabel{$\color{blue}\alpha_0$} [tl] at 15 410
\pinlabel{$d_{+-}$} [tl] at 17 300
\pinlabel{$\Sigma_{\{\Lambda^s\}}$} [tr] at -5 300
\pinlabel{$\color{red}\R \times \Lambda^\ell_1$} [tl] at 105 300
\pinlabel{$\color{blue}b$} [tl] at 25 230
\pinlabel{$d_{+-}^{(-)}$} [tl] at 7 150
\pinlabel{$\R \times \Lambda_0^-$} [tr] at 0 145
\pinlabel{$\color{red}\phi^\ell_{\beta^-\partial_z}(\R \times \Lambda_1)$} [tr] at 300 145
\pinlabel{$\color{blue}A_-$} [tl] at 42 56
\endlabellist
\centering
\includegraphics[height=6cm]{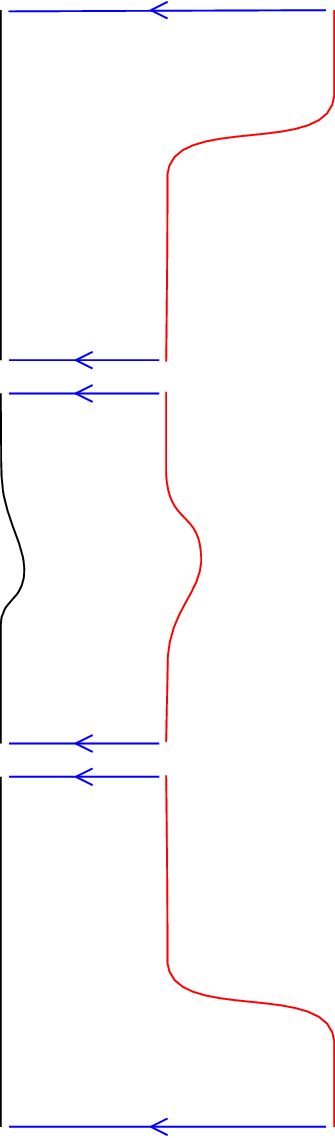}
\caption{Possible limits of pseudoholomorphic discs contributing to $\langle \mathfrak{d}'_{\varepsilon_0^-,\varepsilon_1^-}(A_-),\iota_\ell(\alpha_0)\rangle$ after stretching the neck along $\{0,T\} \times P \times \R$. The middle level is a pseudoholomorphic disc contributing to $\langle \mathfrak{d}_{\varepsilon_0^-,\varepsilon_1^-}(b),\alpha_0\rangle$.}
\label{fig:stretched}
\end{figure}

The proof of Theorem \ref{thm:spectral} can now be finished without much effort by considering the complex $\Cth(\Sigma_{\{\Lambda^s\}},\R \times (\Lambda_1)^\ell))$ with differential $\mathfrak{d}_{\varepsilon_0^-,\varepsilon_1^-}$. The computation in the proof of Lemma \ref{lem:actionreeb} shows that the action of a Reeb chord generator $\gamma_\pm$ from $\Lambda_1^\pm$ to $(\Lambda_0^\pm)^\ell$ is given by
\begin{eqnarray*} && \mathfrak{a}(\gamma_-) = \ell(\gamma_-),\\
&& \mathfrak{a}(\gamma_+)=e^T\ell(\gamma_+)-f^+_{\Sigma_{\{\Lambda^s\}}},
\end{eqnarray*}
for a constant $f^+_{\Sigma_{\{\Lambda^s\}}}>0$. By Part (1) of Lemma \ref{lem:stretch} we see that any representative of
$\alpha_{\varepsilon_0^+,\varepsilon_1^+}(\Lambda_0^+,(\Lambda_1^+)^\ell) \neq 0 \in LCH_{\varepsilon_0^+,\varepsilon_1^+}^*(\Lambda_0^+,\Lambda_1^+)$
must contain a generator of action equal to
$$e^T(c_{\varepsilon_0^+,\varepsilon_1^+}(\Lambda_0^+,\Lambda_1^+)-\ell)-f^+_{\Sigma_{\{\Lambda^s\}}}.$$

We now argue by contradiction, assuming that
$$ c_{\varepsilon_0^-,\varepsilon_1^-}(\Lambda_0^-,\Lambda_1^-)> c_{\varepsilon_0^+,\varepsilon_1^+}(\Lambda_0^+,\Lambda_1^+)$$
holds.  Hence, in view of Lemmas \ref{lem:prestretch} and \ref{lem:stretch}, we can choose a number $\ell>0$ satisfying
$$0<c_{\varepsilon_0^+,\varepsilon_1^+}(\Lambda_0^+,\Lambda_1^+)<\ell<c_{\varepsilon_0^-,\varepsilon_1^-}(\Lambda_0^-,\Lambda_1^-).$$
In other words, every representative of $\alpha_{\varepsilon_0^+,\varepsilon_1^+}(\Lambda_0^+,(\Lambda_1^+)^\ell)$ contains a non-zero multiple of a generator of \emph{negative} action. Since the differential increases action (see Lemma \ref{lem:action}), and since the Reeb chords at the negative end are of positive action, it finally follows that
$$\alpha_{\varepsilon_0^+,\varepsilon_1^+}(\Lambda_0^+,(\Lambda_1^+)^\ell) \notin \im [d_{+-}].$$
This clearly contradicts Part (2) of Lemma \ref{lem:stretch}. \qed

\subsection{The proof of Theorem \ref{thm:liouville} (the case of a contactisation)}
\label{sec:proofliouville}
Write $\Lambda_0:=\Lambda$ and $\Lambda_1:=\Lambda'$.  Let $\Sigma_{\{\Lambda^s\}}$ be the cylinder induced by a positive loop of Legendrians inside $M \setminus \Lambda_1$ starting at $\Lambda^0=\Lambda_0$. By the construction in Section \ref{sec:posit-isot-lagr}, this cylinder may be assumed to be disjoint from the cylinder $\Sigma_1:=\R \times \Lambda_1$.

We argue by contradiction and assume that there are augmentations $\varepsilon_i$, $i=0,1$, of the Chekanov--Eliashberg algebra of $\Lambda_i$ for which $LCC^*_{\varepsilon_0,\varepsilon_1}(\Lambda_0,\Lambda_1)$ is not acyclic.
\begin{Rem}
The non-existence of an arbitrary positive loop of Legendrians containing $\Lambda$ under the stronger assumption that the Legendrian submanifolds have separated $z$-coordinates is an immediate consequence. Namely, in this case we can always arrange so that $\Sigma_{\{\Lambda^s\}}$ and $\R \times \Lambda_1$ become disjoint by translating the latter component sufficiently far in the negative $z$-direction.
\end{Rem}
First observe that the assumption of having a non-vanishing Legendrian contact homology can be translated into the fact that $0<c_{\varepsilon_0,\varepsilon_1}(\Lambda_0,\Lambda_1)<+\infty$ is a finite positive number. Since the length of the Reeb chords from $\Lambda_1$ to $\Lambda_0$ form a discrete subset of $(0,+\infty)$ by the genericity assumptions, after possibly replacing the above augmentations we may even assume that the pair of augmentations is minimal in the sense that
\begin{equation}
\label{eq:ineq}
0<c_{\varepsilon_0,\varepsilon_1}(\Lambda_0,\Lambda_1)\le c_{\varepsilon_0',\varepsilon_1'}(\Lambda_0,\Lambda_1)
\end{equation}
is satisfied for any other pair $(\varepsilon_0',\varepsilon_1')$ of augmentations. 

By Proposition \ref{prp:spectral} we can find augmentations $\varepsilon_i^+$ for which
$$c_{\varepsilon_0^+,\varepsilon_1^+}(\Lambda_0,\Lambda_1)=c_{\varepsilon_0,\varepsilon_1}(\Lambda_0,\Lambda_1)$$
and where $\varepsilon_i^+=\varepsilon_i^- \circ \Phi_{\Sigma_i}$. This, however, is in contradiction with the Inequality \eqref{eq:ineq} combined with the Inequality
$$c_{\varepsilon_0^+,\varepsilon_1^+}(\Lambda_0,\Lambda_1)>c_{\varepsilon_0^-,\varepsilon_1^-}(\Lambda_0,\Lambda_1)$$
established by Theorem \ref{thm:spectral} (here the assumption $\Sigma_0 \cap \Sigma_1 = \emptyset$ is used).
\qed

\section{Wrapped Floer cohomology and non-existence of contractible positive loops}
\label{sec:wrapp-homol-compl}
Here we prove Theorem \ref{thm:filling}, which gives an obstruction to the existence of a contractible positive loop of a Legendrian in terms of the wrapped Floer cohomology of an exact Lagrangian filling. This theory was originally defined in \cite{AbbonFloer} by Abbondandolo--Schwarz and later developed by Abouzaid--Seidel \cite{WrappedFuk} and Ritter \cite{Ritter_TQFT}. 

\subsection{Setup of wrapped Floer cohomology}
Here we give a very brief outline of the definition of wrapped Floer cohomology. We refer to \cite{Ritter_TQFT} as well as \cite{Cie_Oan_ESaxiom} for more details.

Consider a Legendrian $\Lambda \subset (Y,\xi=\ker \alpha)$ living in the contact boundary of a compact Liouville domain $(\overline{X},d\theta)$, and assume that $\Lambda$ admits an exact Lagrangian filling $L \subset (X,d\theta)$ inside the completion of the latter. More precisely, we will assume that
\begin{eqnarray*}
& & (X \setminus \overline{X},d\theta)=((-1,+\infty) \times Y,d(e^t\alpha)),\\
& & L \cap (X \setminus \overline{X})=((-1,+\infty) \times \Lambda,
\end{eqnarray*}
are convex cylindrical ends, while $\overline{L}:=L \cap \overline{X}$ is compact with Legendrian boundary $\partial\overline{L}=\Lambda \subset (Y=\partial \overline{X},\alpha)$.

Now, for each generic $\lambda >0$, consider the autonomous Hamiltonian $H_\lambda \colon X \to \R$ which vanishes in the compact part $\overline{X}$, while it is of the form $\lambda \rho(t)e^t-e^{2T}\lambda$ in the cylindrical end $[-1,+\infty) \times Y$. Here the function $\rho \colon \R \to \R_{\ge 0}$ satisfies $\rho(t)=1$ for $t \ge 2T+\delta$, $\rho(t)=0$ for $t \le 2T-\delta$, and $\frac{d^2}{dt^2}(\rho(t)e^t) \ge 0$ for all $t \in \R$, where $0<\delta<1$ is small. Such a function along with its induced Hamiltonian vector field is schematically depicted in Figure \ref{fig:lambda}.

Given two generic Lagrangian fillings $L_i$, $i=0,1,$ which are cylindrical inside the subset $[T,+\infty) \times Y$, the associated Floer cohomology complex
$$(CF^*(L_0,L_1;H_\lambda),\partial)$$
is now defined as follows.
\begin{itemize}
\item \emph{The generators:} These are the Hamiltonian time-one chords $x(t)$ of $H_\lambda$ from $x(0) \in L_0$ to $x(1) \in L_1$. Equivalently, such chords are intersection points $\phi^1_{H_\lambda}(L_0) \cap L_1$, which moreover can be seen to be of the following two kinds:
\begin{itemize}
\item Intersection points inside $\{ 2T-\delta \le t \le 2T+\delta\}$: these are in bijective correspondence with the Reeb chords from $\Lambda_0$ to $\Lambda_1$ of length at most $\lambda$, and
\item Intersection point inside $\{ t < 2T-\delta \}$: these are simply the intersection points $L_0 \cap L_1$.
\end{itemize}
\item \emph{The differential:} For two generators $x_\pm$, the coefficient $\langle \partial(x_+),x_-\rangle$ of the differential $\partial$ counts the number of rigid finite-energy solutions of the $\overline{\partial}$-equation with Hamiltonian perturbation term
$$
\begin{cases}
u \colon \R \times [0,1] \to X,\\
\partial_s u(s,t)+J_t (\partial_tu(s,t)-X_{H_\lambda}(u(s,t)))=0,\\
u(s,i) \in L_i, \: i=0,1,\\
\lim_{s \to \pm \infty} u(s,t)=x_\pm (t),
\end{cases}
$$
i.e.~so called \emph{Floer strips with boundary on $L_0 \cup L_1$}. The chords $x_+$ and $x_-$ are also called the \emph{input} and \emph{output}, respectively, for obvious reasons.
\end{itemize}
Recall that there are primitives $f_i \colon L_i \to \R$ of the pull-back of the Liouville form $\theta$ to $L_i$, $i=0,1$, by the exactness assumption. For a choice of such primitives, the \emph{action} of a Hamiltonian chord $x(t)=\phi^t_{H_\lambda}$ from $x(0) \in L_0$ to $x(1) \in L_1$ is defined to be
\begin{equation}
\mathcal{A}(x):=f_0(x(0))-f_1(x(1))+\int_0^1 (x^*\theta- H_\lambda (x(t))dt).\label{eq:12}
\end{equation}
It follows that the differential \emph{decreases} the action in the sense that
$$\langle \partial(x_+),x_-\rangle \neq 0$$
implies that $\mathcal{A}(x_-) < \mathcal{A}(x_+)$.

Furthermore, whenever $\lambda \gg 0$ is sufficiently large and $\delta>0$ is sufficiently small (both depending on $L_i$, $i=0,1$), Part (1.b) of Lemma \ref{lem:noescape} below shows that there is a subcomplex
$$(CF^*_0(L_0,L_1;H_\lambda),\partial_0) \subset (CF^*(L_0,L_1;H_\lambda),\partial)$$
consisting of the generators $L_0 \cap L_1 \subset X \setminus ([2T-\delta,+\infty) \times Y).$ Moreover, the quotient complex
$$(CF^*_{+\infty}(L_0,L_1;H_\lambda),\partial_\infty) := (CF^*(L_0,L_1;H_\lambda),\partial)/CF^*_0(L_0,L_1;H_\lambda)$$
has a canonical generating set which is in a canonical bijective correspondence with the set of Reeb chords from $\Lambda_0$ to $\Lambda_1$ of length less than $\lambda$.

The wrapped Floer cohomology is finally defined as the direct limit
$$ HW^*(L_0,L_1) := \lim_{\lambda \to +\infty} HF^* (L_0,L_1;H_\lambda)$$
for a directed system defined by suitable continuation maps. Note that $HW^*(L_0,L_1)$, $L_0=L_1$ is a unital algebra; we refer to \cite{Ritter_TQFT} for the details.

\begin{Rem}
In the cases when it is possible to define the Floer complex $\Cth_*(L_1,L_0)$ of Section \ref{sec:cthulhu-complex} (which with the current technology imposes some constraints), there is an isomorphism
$$HW^*(L_0,L_1) \simeq H(\Cth^*(L_1,L_0))$$
(note the order!) where the right-hand side is the homology of the associated \emph{dual} complex.
\end{Rem}

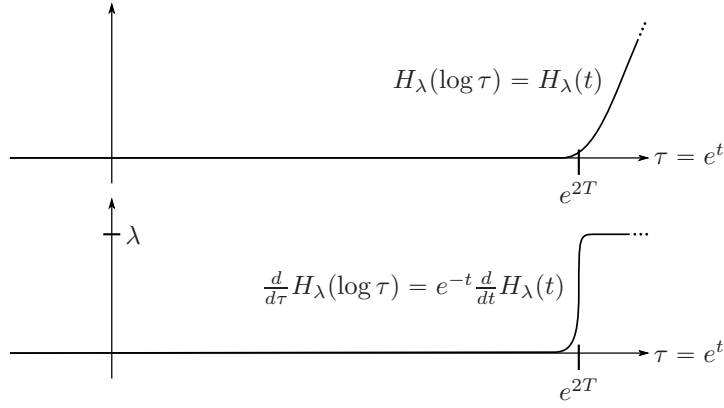
\begin{figure}[htp]
\vspace{1em}
\centering
\labellist
\pinlabel $H_\lambda(\log\tau)=H_\lambda(t)$ at 230 140
\pinlabel $\frac{d}{d\tau}H_\lambda(\log\tau)=e^{-t}\frac{d}{dt}H_\lambda(t)$ at 190 43
\pinlabel $\tau=e^t$ at 320 13
\pinlabel $\tau=e^t$ at 320 106
\pinlabel $\lambda$ at 58 68
\pinlabel $e^{2T}$ at 268 -5
\pinlabel $e^{2T}$ at 268 88
\endlabellist
\includegraphics[scale=0.8]{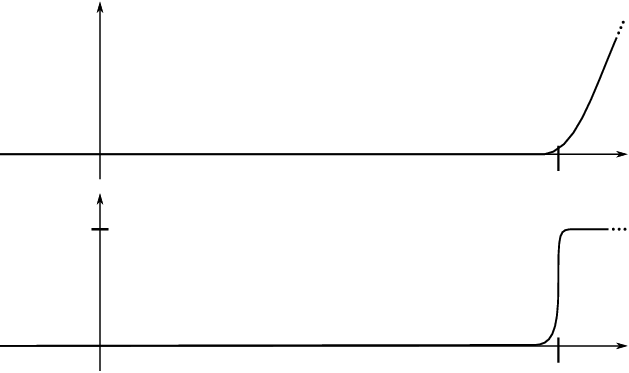}
\vspace{1em}
\caption{Above: The autonomous Hamiltonian $H_\lambda(t)=H_\lambda(\log \tau)$ that is used for wrapping the cylindrical end of the Lagrangian filling $L_0$. Below: the corresponding Hamiltonian vector field is given by $e^{-t}\frac{d}{dt}H_\lambda(t)R_\alpha$, which is parallel to the Reeb vector field.}
\label{fig:lambda}
\end{figure}

\subsection{Proof of Theorem \ref{thm:filling}}

In the following we will make heavy use of the fact that, if on an exact Lagrangian $i:L\rightarrow X$ the Liouville form $\theta$ satisfies $i^*\theta=df$ and if $\phi_G^1$ is the time one flow of a Hamiltonian $G$, then $(\phi^1_G\circ i)^*\theta=d(f+K)$ for the function
 \begin{equation}
  \label{eq:changepotential}
K(q)=\int_0^1 (\theta(X_G)_{\phi^t_G(q)}- H(\phi^t_G(q)))dt
  \end{equation}
(compare with the definition of action in Equation \eqref{eq:12}).

A central technique that also will be used over and over is to use neck stretching in order prevent Floer strips from crossing a given barrier (in the form of a dividing contact hypersurface). There are several different conditions that will be used for this purpose, all which are more or less standard; see e.g.~\cite[Section 2.3]{Cie_Oan_ESaxiom} as well as \cite[Section 6.1]{Dimitroglou:Energy}. For the sake of completeness we here recollect the needed results.

For now we assume that the exact Lagrangian cobordisms $L_i \subset (X,\omega)$, $i=0,1$, both are cylindrical in a neighbourhood of the slice $\{ t_0 \} \times Y$ in the cylindrical end; this is a dividing hypersurface of contact type intersecting $L_i$ in a Legendrian submanifold. We moreover assume that $G_s \colon X \to \R$ is a time-dependent Hamiltonian that vanishes near this slice. The subsets that we consider are $X_L,X_R \subset X$ in the decomposition $X=X_L \cup X_R$ into connected closed subsets such that $X_L \cap X_R = \{ t_0\} \times Y$, i.e.~$X_R=[t_0,+\infty) \times Y$ while $X_L = \overline{X \setminus X_R}$. When here speaking about Floer strips or continuation strips we mean either a Floer strip defined for the time-dependent Hamiltonian $G$, or a continuation strip involving the Hamiltonian that vanishes equivalently and the Hamiltonian $G$. (The latter continuation strips are those appearing in the definition of the continuation maps that turn on or off the Hamiltonian perturbation-term $G$, as well as for the chain homotopies between their compositions.) Also, recall the definition
$$\|G_s\|_{\OP{osc}}:=\int_0^1(\max_X G_s-\min_XG_s)dt \ge 0$$
of the oscillatory norm.
\begin{Lem}
\label{lem:noescape}
Under the above assumptions, and while using primitives $f_i \colon L_i \to \R$ that vanish in the slice $L_i \cap \{t=t_0\}$ in order to define the action, the following is satisfied:
\begin{enumerate}
\item A Floer strip or continuation strip with either
\begin{enumerate}
\item \emph{input} being a generator of \emph{negative} action contained in $X_R$, or
\item \emph{output} being a generator of \emph{positive} action contained in $X_R$,
\end{enumerate}
has both of its asymptotics contained inside $X_R$;
\item 
\begin{enumerate}
\item A Floer strip whose input and output chords are both contained in $X_L$ is contained entirely inside $X_L$, 
\item The same is true for a continuation strip, under the additional assumptions that $G_s$ vanishes inside $X_L$, while its input and output chords $x_{\OP{in}}$ and $x_{\OP{out}}$ satisfies $\mathcal{A}(x_{\OP{out}}) - \mathcal{A}(x_{\OP{in}}) \ge \| G_s \|_{\OP{osc}}$; and
\end{enumerate}
\item Under the additional assumption that $G_s$ vanishes inside $X_R$, it follows that a continuation strip having both input and output contained inside $X_R$ satisfies the property that
\begin{enumerate}
\item the action of the output is not greater than the action of the input, and
\item if its symplectic area moreover vanishes (i.e.~if the actions of the input and output agree), then the strip is contained entirely inside $X_R$ (and is thus constant).
\end{enumerate}
\end{enumerate}
\end{Lem}
\begin{proof}
It was shown in \cite[Lemma 6.2]{Dimitroglou:Energy} that stretching the neck along $\{t_0\} \times Y$ can be interpreted as having the following effect on the action of the generators:
\begin{itemize}
\item the action of a generator in $X_R$ is rescaled by an arbitrarily large positive constant $e^\lambda$, $\lambda \gg 0$, while
\item the action of a generator in $X_L$ is kept fixed.
\end{itemize}
Also, recall that the following basic facts about the symplectic area of Floer strip and continuation strip (see e.g.~\cite[Section 3.2]{Dimitroglou:Energy}). First, an elementary application of Stokes' theorem implies that the symplectic area of either a Floer strip or a continuation strip is given by the action difference $\mathcal{A}(x_{\OP{in}}) -\mathcal{A}(x_{\OP{out}})$ when the input and output asymptotics are the intersection points $x_{\OP{in}}$ and $x_{\OP{out}}$, respectively. The symplectic area moreover satisfies the following:
\begin{enumerate}[label=(\roman*)]
\item The symplectic area of a Floer strip is non-negative, and vanishes if and only if the strip is constant; and
\item For a continuation strip $u\colon \R \times [0,1] \to X$ between the vanishing Hamiltonian and the Hamiltonian $G_s$, the symplectic area of the restriction to any open domain $U \subset \R \times [0,1]$ is bounded from below by the oscillatory norm $-\|G_s\|_{\OP{osc}} \le 0$, under the assumption that $G \circ u$ vanishes on some neighbourhood of $\overline{U} \setminus U$ (we can always take e.g.~$U=\R \times [0,1]$).
\end{enumerate}
For the second property, recall the estimate
$$0 \le \int_U \omega(\partial_tu(s,t),J\partial_tu(s,t))dtds\le \int_Uu^*\,\omega+\|G_s\|_{\OP{osc}}$$
of the Floer energy under the assumptions of the support of $G \circ u$.

(1.a), (1.b), and (3.a): The statements all follow from the above computations and considerations of the action.

(2.a): This is the `no escape lemma' from \cite[Lemma D.6]{Ritter_TQFT}.

(2.b): This is similar to the no escape lemma, but where we first must use the SFT compactness theorem \cite{Bourgeois_&_Compactness} applied to the neck-stretching sequence in order to extract a piece of the strip inside $X_L$ which is of symplectic area greater than $\| G_s \|_{\OP{osc}}$. The existence of such a strip implies the existence of a piece inside $X_R$ with symplectic area smaller than $-\| G_s \|_{\OP{osc}}$, thus contradicting Property (ii) above. (N.B.~here we do not rescale the symplectic form on $X_R$ whilst stretching the neck: we only deform the conformal structure.)

(3.b): This follows from the monotonicity property of the symplectic area of pseudoholomorphic curves with boundary \cite[Propositions 4.3.1 and 4.7.2]{sikorav_monot}, together with the effect on the action by a neck stretching. Namely, the total symplectic area still vanishes after the neck has been stretched, while the symplectic area concentrated near $\{ t_0\} \times Y$ becomes arbitrarily large by the monotonicity property. In other words, the piece of the strip contained inside $X_L$ necessarily has arbitrarily negative symplectic area, which is in contradiction with the bound from Property (ii) above.
\end{proof}

\emph{Step 1:} Write $L_1:=L$. Consider a small push-off $L_0$ of $L$ obtained by, first, applying the small negative Reeb flow
$$\phi^\epsilon_{-e^t} \colon ([-1,+\infty) \times Y,d(e^t\alpha)) \to ([-1,+\infty) \times Y,d(e^t\alpha)), \:\: \epsilon>0,$$
in the cylindrical end and, second, performing a generic Hamiltonian perturbation in the compact part.

We proceed to perturb the Legendrian ends $\Lambda_i = \partial \overline{L}_i \subset Y,$ $i=0,1$, of the fillings. After a small Hamiltonian isotopy of the cylindrical end induced by a contact isotopy, we may assume that $\Lambda_0$ is identified with a section $-j^1f \subset (\mathcal{J}^1\Lambda_1,dz-pdq)$ inside a standard Legendrian neighbourhood in which $\Lambda_1$ is identified with the zero-section. Here the function $f \colon \Lambda_1 \to (0,\epsilon)$ is taken to be a $C^2$-small positive Morse function. We write $\gamma_M^+$ for the Reeb chord from $\Lambda_0$ to $\Lambda_1$ corresponding to the maximum of $f$, which from now on is assumed to be unique.

Consequently, $L_0$ is a section $-dF \subset (T^*L_1,d(pdq))$ in a standard Weinstein neighbourhood of $L_1 \subset (X,d\theta)$ which satisfies $\partial_tF > 0$ outside of a compact subset. We further assume that $F \colon L_1 \to \R$ is a Morse function with no local maximum.

Note that the two potential functions $f_0$ and $f_1$ on the two Lagrangian fillings $L_0$ and $L_1$ can be taken to be $C^2$-close at this step (under the obvious identifications). 

\begin{figure}[htp]
\vspace{1em}
\centering
\labellist
\pinlabel $\color{blue}\widecheck{\gamma}_M$ at 116 41
\pinlabel $\color{blue}\widehat{\gamma}_M$ at 60 41
\pinlabel $\color{blue}\gamma_M^+$ at 256 41
\pinlabel $\tau=e^t$ at 320 9
\pinlabel $L_1=L$ at 315 31
\pinlabel $\color{red}L_0$ at 303 77
\pinlabel $e^1$ at 68 -9
\pinlabel $e^2$ at 88 -9
\pinlabel $e^3$ at 108 -9
\pinlabel $e^{2T}$ at 268 -9
\endlabellist
\includegraphics[scale=0.8]{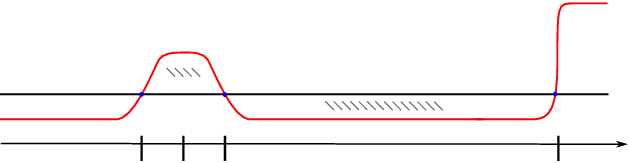}
\vspace{1em}
\caption{A schematic picture of the perturbation of $L_0$ after wrapping by $\phi^1_{H_\lambda}$ together with $L_1$. In this picture the Floer strips contributing to the identities $\partial(\gamma_M^+)=\widecheck{\gamma}_M=\partial(\widehat{\gamma}_M)$ are both visible here.}
\label{fig:smallwrap}
\end{figure}

\begin{Lem}
\label{claim:unit}
For each $\lambda  > 1$ the chain $\gamma_M^+$ given as the Reeb chord from $\Lambda_0$ to $\Lambda_1$ corresponding to the maximum of $f$ is a cycle whose limit as $\lambda \to +\infty$ represents the unit in $HW^*(L,L)$.
\end{Lem}
\begin{proof}
Partially wrap the end of $L$ so that the chords corresponding to critical points of $f$ becomes intersection points contributing to $CF_0(L,L,H_\lambda)$, the proposition follows now from the description of the unit in \cite[Section 6.3 and 6.13]{Ritter_TQFT}. Note that this generator is the global maximum of a Morse function on $L_1$, the differential of which has graph identified with $L_0$.
\end{proof}

\emph{Step 2:}  Inside the cylindrical part $X \setminus \overline{X}$ there exists a Weinstein neighbourhood identifying $L_1$ with the zero-section in $(T^*L_1,d(pdq)),$ while $L_0$ is given by the graph of $-d(e^tf)$. Here the coordinate $t \colon L_1 \to \R$ is induced by the coordinate on the $\R$-factor of the symplectisation $\R \times Y$. We refer to \cite[Section 8.4]{cthulhu} for the construction of such a Weinstein neighbourhood.

Now we perform the following modification inside the subset $[0,4] \times Y$. Consider a Morse function $g \colon [-1,+\infty) \to \R$ satisfying $g(t)=t$ for $t \ge 4$ as well as for $t \le 0$, while $g'(1)=g'(3)=0$, $g''(1)<0$, $g''(3)>0$, are its unique critical points. I.e.~$g(t)$ has a non-degenerate local maximum at $t=1$ and a non-degenerate local minimum at $t=3$. We moreover require that $g'(t) \equiv -C < 0$ is constant in the subset $\{ 2-2\delta \le t \le 2+2\delta\}$.

We replace $L_0$ by the graph $-d(e^{g(t)}f)$ and again denote the resulting Lagrangian cobordism by $L_0$. There are now additional intersection points of $L_0 \cap L_1$ contained in the slices $\{t =1\}$ and $\{ t =3 \}$. In particular, we have the unique local maximum $\widehat{\gamma}_M$ of $e^{g(t)}f$ contained in the slice $\{ t = 1\}$ and the unique critical point $\widecheck{\gamma}_M$ of index $\dim L-1$ contained in the slice $\{ t = 3\}$. If $g(t)-t$ is chosen to be sufficiently $C^0$-small, it follows that the new potential function $f_0$ is $C^2$-close to the original one. C.f.~Figure \ref{fig:smallwrap}.

\begin{Lem}
\label{claim:unit2}
For each $\lambda > 1$ the chain $\widehat{\gamma}_M+\gamma_M^+$ is a cycle whose limit as $\lambda \to +\infty$ represents the unit in $HW^*(L,L)$. Moreover, $\partial(\gamma_M^+)=\widecheck{\gamma}_M=\partial(\widehat{\gamma}_M)$ is satisfied. (See Figure \ref{fig:smallwrap}.)
\end{Lem}
\begin{proof}
The first statement is the same as Lemma \ref{claim:unit} as now the unit in $HF_0(L_0,L_0')$ where $L'_0$ is the partially wrapped Lagrangian is given by the sum of intersection points coming from $\widehat{\gamma}_M$ and $\gamma_M^+$. The second statement follows from the explicit description of holomorphic curves in the Weinstein neighbourhood similarly to Floer's computation in \cite{FloerHFlag} of the differential in the case of a small Hamiltonian push-off.
\end{proof}

\begin{figure}[htp]
\vspace{1em}
\centering
\labellist
\pinlabel $h(\log\tau)=h(t)$ at 155 148
\pinlabel $\frac{d}{d\tau}h(\log\tau)=e^{-t}\frac{d}{dt}h(t)$ at 163 55
\pinlabel $\tau=e^t$ at 320 13
\pinlabel $\tau=e^t$ at 320 106
\pinlabel $e^{2-\delta}$ at 68 -5
\pinlabel $e^2$ at 88 -5
\pinlabel $e^{2+\delta}$ at 108 -5
\pinlabel $e^{2-\delta}$ at 68 88
\pinlabel $e^2$ at 88 88
\pinlabel $e^{2+\delta}$ at 108 88
\pinlabel $1$ at 58 68
\endlabellist
\includegraphics[scale=0.8]{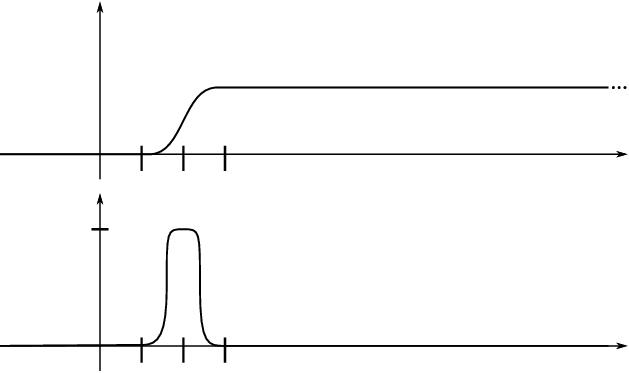}
\vspace{1em}
\caption{Above: The autonomous Hamiltonian $h(t)=h(\log \tau)$ that is used to perform the `finger move' in Step 3. Below: the corresponding Hamiltonian vector field is given by $e^{-t}\frac{d}{dt}h(t)R_\alpha$, which is parallel to the Reeb vector field.}
\label{fig:kappa}
\end{figure}
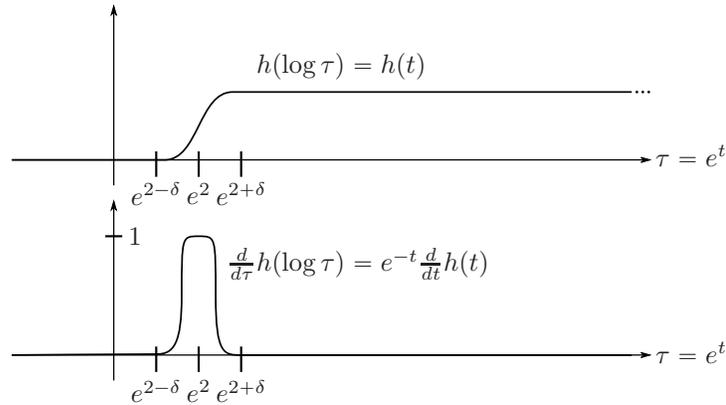

\emph{Step 3:} We use a $C^0$-small Hamiltonian to introduce a very long `finger move' at $t=2$.  More precisely, we apply a Hamiltonian isotopy of the form $\phi^\kappa_h$ for $\kappa \gg 0$, where the Hamiltonian $h(t)$ satisfies the property that $e^{-t}\frac{d}{dt}h(t) \ge 0$ is a bump-function supported inside $(2-\delta,2+\delta) \subset [-1,+\infty)$ and which is constantly equal to $e^{-t}\frac{d}{dt}h(t)\equiv 1$ near $t=2$. Moreover, we assume that $\frac{d}{dt}(e^{-t}\frac{d}{dt}h(t)) \le 0$ and $\frac{d}{dt}(e^{-t}\frac{d}{dt}h(t)) \ge 0$ holds on $t \ge 2$ and $t \le 2$, respectively. Such a bump function is shown in Figure \ref{fig:kappa}. We denote the resulting filling by $L_0^\kappa$. Note that it follows from Equation \eqref{eq:changepotential} that for $\kappa \gg 0$, the newly created intersection points all correspond to Reeb chords from $\Lambda_0$ to $\Lambda_1$ of length longer than the chords corresponding $\OP{Crit}(f)$; these chords themselves correspond to Reeb chords on $\Lambda$ being of length at most $\kappa$. More precisely, for each such Reeb chord $\gamma$ on $\Lambda$ of length $\ell(\gamma) \le \kappa$ there is a corresponding pair of intersection points $\gamma_a \in \{ 2-\delta < t < 2 \}$ and $\gamma_b \in \{ 2<t<2+\delta\}$, both being of action roughly equal to $e^2\ell(\gamma)$.

Part (1.a) of Lemma \ref{lem:noescape} applied to $\{2 \} \times Y$ then shows that the Floer differential $\partial_0$ satisfies the property that the generators contained in $[2,+\infty) \times Y$ form a subcomplex
$$(CF_{\OP{rel}}(L_0^\kappa,L_1;H_\lambda),\partial_{\OP{rel}}) \subset (CF(L_0^\kappa,L_1;H_\lambda),\partial)$$
whenever $\kappa \gg 0$ is sufficiently large. This is related to the definition of relative wrapped Floer cohomology in \cite{Cie_Oan_ESaxiom} as well as the construction of the Viterbo transfer map. By elementary action reasons, we also have a subcomplex
\begin{eqnarray*}
& & CF_{\OP{rel},\le 0}(L_0^\kappa,L_1;H_\lambda),\partial_{\OP{rel},\le 0}) \subset (CF_{\OP{rel}}(L_0^\kappa,L_1;H_\lambda),\partial_{\OP{rel}}),\\
& & CF_{\OP{rel},\le 0}(L_0^\kappa,L_1;H_\lambda):=CF_{\OP{rel}}(L_0^\kappa,L_1;H_\lambda) \cap CF_{0}(L_0^\kappa,L_1;H_\lambda).
\end{eqnarray*}
generated by intersection points $L_0 \cap L_1$ as well as those intersection points of $L^\kappa_0 \cap L_1$ contained inside $\{ 2 < t < 3 \}$, i.e.~corresponding to Reeb chords from $\Lambda_0$ to $\Lambda_1$ being of length at most $\kappa$.
\begin{Lem}
\label{claim:boundary}
For any $\kappa \ge 0$ we still have
$$\widecheck{\gamma}_M =\partial(\gamma_M^+)=\partial_{\OP{rel}}(\gamma_M^+)=\partial(\widehat{\gamma}_M)=\partial_0(\widehat{\gamma}_M)\in (CF_{\OP{rel}}(L_0^\kappa,L_1;H_\lambda),\partial_{\OP{rel}}).$$
Moreover, if the cycle $\widecheck{\gamma}_M$ is a $\partial_{\OP{rel},\le 0}$-boundary for some $\kappa \ge 0$, then $HW^*(L,L)=0$.
\end{Lem}
\begin{proof}
Since the Hamiltonian $h \colon X \to \R$ used to perform the finger move is $C^0$-small, the continuation map $\phi_\kappa$ induced by $h$ may be assumed to be action decreasing. Recall that
\begin{gather*}
\phi_\kappa \colon CF(L_0,L_1;H_\lambda) \xrightarrow{\sim} CF(L_0^\kappa,L_1;H_\lambda),\\
\phi_\kappa|_{CF_0(L_0,L_1;H_\lambda)} \colon CF_0(L_0,L_1;H_\lambda) \xrightarrow{\sim} CF_0(L_0^\kappa,L_1;H_\lambda)
\end{gather*}
is a chain homotopy equivalence defined by counting `continuation' strips. Here, Part (1.b) of Lemma \ref{lem:noescape} has been applied to the slice $\{2T-\delta\} \times Y$ in order to infer that also $\phi_\kappa|_{CF_0(L_0,L_1;H_\lambda)}$ is a homotopy equivalence.

Moreover, we have the identities
\begin{align*}
& \phi_\kappa(\widehat{\gamma}_M)=\widehat{\gamma}_M+b, \; b\in CF_{\text{rel},\leq 0}(L_0^\kappa,L_1;H_\lambda)\\
& \phi_\kappa(\widecheck{\gamma}_M)=\widecheck{\gamma}_M,\\
& \phi_\kappa(\gamma_M^+)=\gamma_M^+.
\end{align*}
To see this, recall that $h$ vanishes near the generators $\widehat{\gamma}_M$, $\widecheck{\gamma}_M$ and $\gamma_M^+$ by construction. These three equalities now follow from Lemma \ref{lem:noescape}, by inferring that only the constant strips give contributions. More precisely, the first equality is shown by applying Part (2.b) to $\{ 2-\delta\} \times Y$, while the second and third equalities are shown by applying Part (3.b) to $\{ 2+\delta\} \times Y$. Combining these three equalities with the explicit computation made in Lemma \ref{claim:unit2}, we can now conclude the first claim.

Now assume that $\partial_{\OP{rel},\le 0}(a)=\widecheck{\gamma}_M$, where the action of $a$ thus can be assumed to be significantly larger than that of both $\widecheck{\gamma}_M$ and $\widehat{\gamma}_M$. By the above, we now compute $\partial_0(\widehat{\gamma}_M-a)=0$, i.e.~$\widehat{\gamma}_M-a$ is a nontrivial cycle. However, since no chain of the form $\widehat{\gamma}_M-a+\partial_0(b)$ is in the image of $\phi_\kappa$ by its action-decreasing properties, and since $\phi_\kappa|_{CF_0(L_0,L_1;H_\lambda)}$ is a chain homotopy equivalence, the cycle $\widehat{\gamma}_M-a$ must be a $\partial_0$-boundary. In particular, there exists a chain $c$ for which $\partial_0(c)=\widehat{\gamma}_M$ holds modulo an element in $CF_{\OP{rel}}(L_0^\kappa,L_1;H_\lambda)$.

Note that the quotient complex $(CF_0(L_0^\kappa,L_1;H_\lambda),\partial_0)/CF_{\OP{rel}}(L_0^\kappa,L_1;H_\lambda)$ again can be used to compute $HW^*(L,L)$, when taking the appropriate direct limit $\kappa \to +\infty$. C.f.~the definition of the Viterbo transfer map whose construction goes via such a quotient. (In this case we are computing the Viterbo transfer of a trivial cobordism, which hence gives an isomorphism). Since the limit of $\widehat{\gamma}_M$ becomes the unit of the algebra $HW^*(L,L)$, the statement now follows.
\end{proof}

\emph{Step 4:} Replace $L_0^\kappa \cap ([5,+\infty) \times Y)$ with the concordance $\Sigma_{\{\Lambda^s\}}$ induced by a positive loop, thereby producing the exact Lagrangian filling $\widetilde{L}_0^\kappa$ which is Hamiltonian isotopic to $L_0$ for a Hamiltonian having compact support contained inside $(5,T) \times Y$ (here we used the assumption that the isotopy is contractible together with Proposition \ref{prp:contractible}). Here it is necessary that the constant $T \ge 0$ chosen in the initial setup for the computation of the wrapped Floer cohomology is sufficiently large. Recall that the constant $T>0$ was taken so that all our data is cylindrical inside the subset $[T,+\infty)\times Y$.

The invariance proof for the Floer complex under compactly supported Hamiltonian isotopies produces a chain homotopy equivalence
\[\phi \colon (CF(L_0^\kappa,L_1;H_\lambda),\partial) \to (CF(\widetilde{L}_0^\kappa,L_1;H_\lambda),\partial').\]
For $\kappa \gg 0$ sufficiently large, Lemma \ref{lem:noescape} again shows that
\begin{eqnarray*}
& & \phi(CF_0(L_0^\kappa,L_1;H_\lambda)) \subset CF_0(\widetilde{L}_0^\kappa,L_1;H_\lambda),\\
& & \phi(CF_{\OP{rel}}(L_0^\kappa,L_1;H_\lambda)) \subset CF_{\OP{rel}}(\widetilde{L}_0^\kappa,L_1;H_\lambda)
\end{eqnarray*}
are satisfied; for the first statement we apply Part (1.b) of this lemma to $\{2T-\delta\} \times Y$ while for the second statement we apply Part (1.a) to the slice $\{2\} \times Y$.

\begin{Lem}
\label{claim:action}
The Reeb chord $\gamma_M^+ \in (CF^*(\widetilde{L}_0^\kappa,L_1;H_\lambda),\partial')$ is a $\partial'$-cycle which can be assumed to be of negative action. Moreover, $\widecheck{\gamma}_M \in CF^*(\widetilde{L}_0^\kappa,L_1;H_\lambda)$ is of positive action, while all the generators corresponding to the intersection points $\widetilde{L}_0^\kappa \cap L_1 \cap \{t \ge 5\}$ are of negative action.
\end{Lem}
\begin{proof}
This is a straight-forward action computation; c.f.~the computation made in Lemma \ref{lem:actionreeb}.
\end{proof}

\begin{Lem}
\label{claim:main}
If some chain of the form $\widecheck{\gamma}_M+c \in CF^*_{\OP{rel},\le 0}(\widetilde{L}_0^\kappa,L_1;H_\lambda)$ is a $\partial_{\OP{rel},\le 0}'$-boundary, where $c$ is a linear combination of generators of negative action, then $HW(L,L)=0$.
\end{Lem}
\begin{proof}
By the action properties established in Lemma \ref{claim:action} together with the fact that the differential decreases action, we clearly have
$$\langle \partial_{\OP{rel}}'(c'),\widecheck{\gamma}_M \rangle=0$$
whenever $c'$ has negative action. The generators of $CF^*_{\OP{rel},\le 0}(\widetilde{L}_0^\kappa,L_1;H_\lambda)$ being of non-negative action are contained inside $\{ t \le 4 \}$. An application of Part (2.a) of Lemma \ref{lem:noescape} to $\{4 \} \times Y$ now shows that any Floer trajectory contributing to $\langle \partial_{\OP{rel}}'(a),\widecheck{\gamma}_M \rangle=1$ in fact must live entirely inside the same subset. These Floer trajectories are thus in bijective correspondence with the Floer trajectories corresponding to $\langle \partial_{\OP{rel},\le 0}(a),\widecheck{\gamma}_M\rangle$. The result is then deduced from Lemma \ref{claim:boundary}.
\end{proof}

\begin{Lem}
\label{claim:invariance}
We have
$$\phi(\gamma_M^+)=\gamma_M^++b, \:\: b \in CF^*_{\OP{rel},\le 0}(\widetilde{L}_0^\kappa,L_1;H_\lambda)$$
while, for $\kappa \gg 0$ sufficiently large, we also have
$$\phi(\widecheck{\gamma}_M)=\widecheck{\gamma}_M+\partial'_{\OP{rel},\le 0}(c)+d, \: \:c,d \in CF^*_{\OP{rel},\le 0}(\widetilde{L}_0^\kappa,L_1;H_\lambda),$$
where $d$ moreover is a sum of generators of \emph{negative} action.
\end{Lem}
\begin{proof}
For the first claim, we argue as follows. Applying Part (3.a) of Lemma \ref{lem:noescape} to $\{2T-\delta \} \times Y$, it follows that the continuation map must decrease the action when restricted to the Reeb chord generators in $\{ t \ge 2T-\delta\}$. Since $\gamma_M^+$ is the shortest Reeb chord from $\Lambda_0$ to $\Lambda_1$, it is of least action amongst the generators of the quotient
$$CF(L_0^\kappa,L_1;H_\lambda) / CF_0(L_0^\kappa,L_1;H_\lambda).$$
What remains is thus now to show that $\langle \phi(\gamma_M^+),\gamma_M^+ \rangle =1$. This count is established by inferring that a continuation strip which contributes to this count must be confined to the subset $\{ t \ge 2T-\delta \}$ where the Hamiltonian vanishes. We can then use the standard fact that a rigid continuation strip must be constant for a vanishing Hamiltonian. That the strip is contained in the region is the case by Part (3.b) of Lemma \ref{lem:noescape} applied to $\{ 2T-\delta\} \times Y$.

We now continue with the second claim. Part (2.a) of Lemma \ref{lem:noescape} applied to $\{4\} \times Y$ shows that
\begin{equation}
\label{eq:noescape}
\langle \partial_0'(a_1),a_2 \rangle =\langle \partial_0(a_1),a_2 \rangle
\end{equation}
whenever $a_i$, $i=1,2$, are generators contained inside $\{ t < 4 \}$. Namely, the Floer strips contributing to these counts must be confined to the subset $\{ t \le 4 \}$ in which $\widetilde{L}_0^\kappa \cap \{ t \le 4 \} = L_0^\kappa \cap \{ t \le 4 \}$ is satisfied.

If, in addition, $\kappa \gg 0$ is taken sufficiently large, a further application of Part (2.b) of Lemma \ref{lem:noescape} to $\{2\} \times Y$ shows that
$$\phi(\widehat{\gamma}_M)=\widehat{\gamma}_M+c, \:\: c \in CF_{\OP{rel},\le 0}(\widetilde{L}_0^\kappa,L_1;H_\lambda)$$
holds as well. Indeed, in this case the strips contributing to $\langle \phi(a_1),a_2 \rangle$, for generators $a_i$, $i=1,2$, contained inside $\{ t < 2 \}$, must be contained entirely inside the subset $\{ t < 2\}$ where the Hamiltonian vanishes. Again, such strips are hence constant.

The chain map property now gives
$$ \phi(\widecheck{\gamma}_M)=\phi(\partial_0(\widehat{\gamma}_M))=\partial'_0(\phi(\widehat{\gamma}_M))=\partial'_0(\widehat{\gamma}_M+c)=\partial_0(\widehat{\gamma}_M) + d + \partial'_{\OP{rel},\le 0}(c)$$
where we rely on Lemma \ref{claim:boundary} for the first equality and Equality \eqref{eq:noescape} for the last equality.
\end{proof}

{\em Step 5:} We are now ready to finish the proof of Theorem \ref{thm:filling}. From the chain map property together with Lemma \ref{claim:invariance} we see that
$$\partial'(\gamma_M^++b)=\partial'(\phi(\gamma_M^+))=\phi(\partial(\gamma_M^+))=\phi(\widecheck{\gamma}_M)=\widecheck{\gamma}_M+\partial'_{\OP{rel},\le 0}(c)+d,$$
where $b,c,d \in CF_{\OP{rel},\le 0}(\widetilde{L}_0^\kappa,L_1;H_\lambda)$, and $d$ is a sum of generators of negative action. Since
$$\langle \partial'(a),\widecheck{\gamma}_M\rangle=\langle \partial'_{\OP{rel},\le 0}(a),\widecheck{\gamma}_M\rangle=0, \:\: \forall a \in CF_{\OP{rel},\le 0}(\widetilde{L}_0^\kappa,L_1;H_\lambda),$$
holds by Lemma \ref{claim:main} together with the assumption that $HW^*(L,L) \neq 0$, we now conclude that necessarily $\langle \partial'(\gamma_M^+),\widecheck{\gamma}_M \rangle \neq 0$. This, however, is in contradiction with the action computation in Lemma \ref{claim:action}. In other words, the hypothetical contractible positive Legendrian isotopy containing $\Lambda_0$ cannot exist. \qed

\section{Applications to strong orderability}

In this paragraph, we apply our techniques to the study of strong orderability in the sense of Liu \cite{Liu_paper}: we prove Theorem~\ref{thm: strong} by using Theorem \ref{thm:filling}.

The proof of Theorem~\ref{thm: strong} relies on the following equivalence, known to experts:

\begin{Thm}\label{thm: equivalence} Let $(W,\omega=d\alpha )$ be a Liouville domain and denote by  $(\widehat{W} ,\widehat{\omega})$ its completion by addition of the positive half symplectisation $([0,+\infty) \times \partial W,d(e^s\alpha))$ of $(\partial W,\alpha)$ along $\partial W$. Let  $\Delta_{\widehat{W}\times \widehat{W}}$ be the Lagrangian diagonal in the symplectic product $(\widehat{W}\times \widehat{W}, \widehat{\omega}\oplus -\widehat{\omega})$. Then the wrapped Floer cohomology of $\Delta_{\widehat{W}\times \widehat{W}}$ is isomorphic to the symplectic homology of $(W,d\alpha)$.
\end{Thm}
\begin{proof} We sketch a proof following closely the lines proposed by Zena\"\i di.
We start from a time-dependant Hamiltonian function $H:\R \times \widehat{W}\to \R$ which equals to (a perturbation of) $e^{2s}$ in $[0,+\infty) \times \partial W$, $s\in [0,+\infty)$.
On $\widehat{W}\times \widehat{W}$, we consider the split Hamiltonian $H_\oplus :\R \times \widehat{W}\times \widehat{W} \to \R$ defined by $H_\oplus (t,x,y) = H(t,x)+H(t,y)$, as well as a split almost complex structure $J_\oplus =J\oplus (-J)$ compatible with $\widehat{\omega}\oplus -\widehat{\omega}$.
With these data, we define a relative symplectic homology, by counting Floer strips in  $(\widehat{W}\times \widehat{W}, J_\oplus, H_\oplus)$ with boundary on $\Delta_{\widehat{W}\times \widehat{W}}$ between time-1 chords of the Hamiltonian $\Phi_{H_\oplus}$. Notice here that a Hamiltonian chord from $(x,x)$ to $(y,y)$ consists of H-chords from $x$ to $y$ on the first factor and from $y$ to $x$ on the second factor.

We first show that this homology is isomorphic to $SH(W)$.
For that, we let $\tau : (\R \times [0,1],i)\to (\R\times [0,1],i)$ be the anti-holomorphic involution of the strip  given by $\tau (s,\theta)=(s,1-\theta)$. It switches the two boundary components.
If $u:(\R\times [0,1],i)\to (\widehat{W}\times \widehat{W}, J_\oplus, H_\oplus)$ is a Floer strip, then its projection/twisted projection to the first and second factors $u_1 =\pi_1 \circ u$ and $u_2 =\pi_2 \circ u\circ \tau$ satisfy the Floer equation in $(\widehat{W} ,J,H)$. Moreover, since $u(s,1)=(u_1(s,1),u_2(s,0))$ and $u(s,0)=(u_1(s,0),u_2(s,1))$ belong to $\Delta_{\widehat{W}\times \widehat{W}}$, one has that $u_1(s,1)=u_2(s,0)$ and $u_1(s,0)=u_2(s,1)$. This means that we can glue the two strips $u_1$ and $u_2$ together along their boundary components to obtain a Floer cylinder $u_1 \sharp u_2 : \R \times S^1 \to (\widehat{W},J,H)$ which is exactly of the type counted by the differential in symplectic Floer homology of $(\widehat{W},J,H)$ (defined by the time-2 periodic orbits of the Hamiltonian $H$).
Conversely, if we parametrise the circle $S^1$ by $\R /(2\Z)$, any cylinder $u :\R \times S^1 \to (\widehat{W},J,H)$ counted in the differential of the Floer complex of $(\widehat{W},J,H)$ can be decomposed in two strips
$u_1 :\R \times [0,1] \to (\widehat{W},J,H)$ and $u_2 :\R \times [1,2] \to (\widehat{W},J,H)$ with matching boundary conditions. We reparametrise $u_2$ to $u_2' (s,\theta )=u_2 (s,\theta -1)$.
The map $u=(u_1,u_2'\circ \tau) : (\R\times [0,1],i)\to (\widehat{W}\times \widehat{W}, J_\oplus, H_\oplus)$ is then a Floer strip with boundary on  $\Delta_{\widehat{W}\times \widehat{W}}$.

To conclude, it remains to deform the Hamiltonian data $(J_\oplus ,H_\oplus)$ to one $(\mathcal{J} ,\mathcal{H})$ needed to define the wrapped complex. This can be done amongst Hamiltonians with exponential growth in the direction of the Liouville vector field in $(\widehat{W}\times \widehat{W}, \widehat{\omega}\oplus -\widehat{\omega})$. Such a path of data induces an isomorphism at the homology level given by continuation maps defined by counts of strips satisfying parametrised Floer equations, as carried out by Oancea in his thesis \cite{Oan}. We also recall the standard fact that the homology counting Hamiltonian self chords on $\Delta_{\widehat{W}\times \widehat{W}}$ and Floer strips with boundary on $\Delta_{\widehat{W}\times \widehat{W}}$ is isomorphic to the Lagrangian Floer homology counting intersection points between $\Delta_{\widehat{W}\times \widehat{W}}$ and $\Phi_{H_\oplus} (\Delta_{\widehat{W}\times \widehat{W}})$ as well as holomorphic strips.
\end{proof}

We now complete the proof of Theorem \ref{thm: strong}.
\begin{proof}
We consider the completion $(\widehat{W} ,\widehat{\omega})$ of $(W,\omega)$ obtained by the addition of the positive half symplectisation $([0,+\infty) \times M,d(e^s\alpha))$ of $(M,\alpha)$ to $(W,\omega)$ along $M=\partial W$. Let $\phi_1$ be a Hamiltonian diffeomorphism of $(\widehat{W} ,\widehat{\omega})$ whose Hamiltonian $H$ equals $e^{2s}$ in $[0,+\infty) \times M$. The symplectic homology of $(W,\omega )$ is the Hamiltonian Floer homology of $\phi_1$.
By Theorem~\ref{thm: equivalence}, it is the wrapped Floer cohomology of the diagonal $\Delta_{\widehat{W}\times \widehat{W}}$ in $(\widehat{W} \times \widehat{W} , \omega \oplus -\omega)$. Note that 
$(\widehat{W} \times \widehat{W} , \omega \oplus -\omega)$ has an ideal contact boundary $(V,\zeta)$ in which the Lagrangian $\Delta_{\widehat{W}\times \widehat{W}}$ has an ideal Legendrian boundary $L$.
The diagonal $\Delta_{\widehat{W}\times \widehat{W}}$
is a Lagrangian filling of the Legendrian $L$.
From Theorem \ref{thm:filling}, we get that $L$ is not the basepoint of a contractible positive loop in $(V,\zeta)$. To conclude it remains to observe that the contact product $(M\times M\times \R ,\alpha_1 -e^t \alpha_2 )$ is a contact submanifold of $(V,\zeta)$ which contains $L$: it is a standard neighbourhood of $\partial W\times \partial W \subset V$ in $(V,\zeta )$.
\end{proof}

\bibliographystyle{plain}
 \bibliography{Bibliographie_en}
\end{document}